\documentclass[12pt]{amsart}
\usepackage{amsmath,amssymb,latexsym, amsthm, amscd, stmaryrd, amsfonts, mathrsfs}
\usepackage[title,titletoc,toc]{appendix}
\usepackage[all]{xy}
\usepackage{soul}

\theoremstyle{definition}
\newtheorem{rem}[subsubsection]{Remark}
\theoremstyle{plain}
\newtheorem{prop}[subsubsection]{Proposition}
\newtheorem{thm}[subsubsection]{Theorem}
\newtheorem{lem}[subsubsection]{Lemma}
\newtheorem{cor}[subsubsection]{Corollary}

\newcommand{\mbf}{\mathbf}
\newcommand{\mfk}{\mathfrak}
\newcommand{\mscr}{\mathscr}
\newcommand{\mcal}{\mathcal}
\newcommand{\mbb}{\mathbb}
\newcommand{\mrm}{\mathrm}
\newcommand{\A}{\mathfrak a}
\newcommand{\Aa}{\mathfrak a}
\newcommand{\B}{\mathfrak b}
\newcommand{\C}{\mathfrak c}
\newcommand{\D}{\mathfrak  d}

\newcommand{\Hd}{\mbf H_d}

\newcommand{\M}{\mathfrak m}

\newcommand{\z}{\mathfrak z}

\newcommand{\ma}{\mathfrak a}
\newcommand{\mb}{\mathfrak b}
\newcommand{\mc}{\mathfrak c}
\newcommand{\md}{\mathfrak d}

\newcommand{\m}{\mathfrak m}
\newcommand{\p}{ {\rm  p}}
\newcommand{\ro}{\mrm{ro}}
\newcommand{\co}{\mrm{co}}
\newcommand{\ur}{{\rm ur}}
\newcommand{\wt}{\widetilde}
\newcommand{\sgn}{{\rm sgn}}

\newcommand{\T}{\tau}
\newcommand{\U}{\mbf U}

\newcommand{\End}{\mrm{End}}

\usepackage{color}
\newcommand{\nc}{\newcommand}
\nc{\redtext}[1]{\textcolor{red}{#1}}
\nc{\bluetext}[1]{\textcolor{blue}{#1}}
\nc{\greentext}[1]{\textcolor{green}{#1}}
\nc{\yl}[1]{\redtext{From yq: #1}}
\nc{\zb}[1]{\redtext{From zb: #1}}

\title[Geometric Schur Duality, II]{Geometric Schur Duality of  Classical Type, II}

\author{Zhaobing Fan}
\author{Yiqiang Li}
\address{Department of Mathematics\\ University  at Buffalo, SUNY\\Buffalo, NY 14260}
\email{zhaobing@buffalo.edu, yiqiang@buffalo.edu}

\date{\today}

\begin{document}

\begin{abstract}
We establish  algebraically and geometrically  a duality between the  Iwahori-Hecke algebra of type $\mbf D$ and two  new  quantum algebras arising from the geometry of $N$-step isotropic  flag varieties of type $\mbf D$.
This  duality is  a type $\mbf D$ counterpart  of  the Schur-Jimbo duality of type $\mbf A$ and the Schur-like duality of type $\mbf B/\mbf C$ discovered by  Bao-Wang.
The new  algebras play  a role in the type $\mbf D$ duality  similar to the modified  quantum $\mathfrak{gl}(N)$ in type $\mbf A$, and the modified coideal subalgebras of 
quantum $\mathfrak {gl}(N)$ in type $\mbf B/\mbf C$.
We construct canonical bases for  these two  algebras.
\end{abstract}

\maketitle

\tableofcontents

\section{Introduction}

Let $G$ be a classical linear algebraic group over an algebraically closed field.
One of the milestones  in geometric representation theory 
is the geometric  realization of the associated Iwahori-Hecke algebra of $G$, by using 
the bounded derived category of  $G$-equivariant constructible sheaves on the  product variety of two copies  of the complete flag variety of $G$. 
Via this realization, many problems related to the Iwahori-Hecke algebra of $G$ are solved. For example,  
the positivity conjecture for the structure constants of  the Kazhdan-Lusztig bases (\cite{KL79}) are settled by interpreting the basis elements as the (shifted)  intersection cohomology complexes attached to  $G$-orbit closures in the product variety.

One may wonder if the geometric approach can be adapted to study other objects in representation theory, besides Iwahori-Hecke algebras. 
Indeed, a  modification by replacing the adjective `complete'  in the construction by `partial'  already yields highly nontrivial results, as we explain in the following.

If $G$ is of type $\mbf A$, i.e., $G=\mrm{GL}(d)$, and  the complete flag variety is replaced  by the  $N$-step partial flag variety of $\mrm{GL}(d)$ with $N$ bearing not relation to $d$,  then an  analogous construction  provides  a geometric realization of the $v$-Schur quotient  of the  quantum $\mathfrak{ gl}(N)$  in the classic work ~\cite{BLM90}.
Moreover,  the quantum $\mathfrak{gl}(N)$ can then be realized in the projective limit of the $v$-Schur quotients (as $d$ goes to infinity). 
Remarkably, an idempotented version of quantum $\mathfrak{gl}(N)$ is  discovered inside the projective limit as well admitting a canonical basis.
The role of the canonical basis for  the modified quantum $\mathfrak{gl}(N)$ is similar to that  of  Kazhdan-Lusztig bases  for  Iwahori-Hecke algebras.
Subsequently, the Schur-Jimbo duality,
as a bridge  connecting the Iwahori-Hecke algebra of $\mrm{GL}(d)$ and (modified) quantum $\mathfrak{gl}(N)$,
is realized geometrically by considering the product variety of the  complete flag variety and the $N$-step partial flag variety of $\mrm{GL}(d)$ in ~\cite{GL92}.
The modified quantum $\mathfrak{sl}(N)$ (a variant  of quantum $\mathfrak{gl}(N)$)  and its canonical basis are further categorified in the works  ~\cite{L10} and  ~\cite{KL10}, which play a fundamental role in higher representation theory and the categorification of knot invariants.

If $G$ is of type $\mbf B /\mbf C$, i.e., $G=\mrm{SO}(2d+1)/\mrm{SP}(2d)$, and the  variety involved  is replaced
by the  $N$-step isotropic flag variety of $\mrm{SO}(2d+1)/\mrm{SP}(2d)$, then one gets a geometric realization of the modified forms of
 two coideal subalgebras $\U^{\imath}$ and $\U^{\jmath}$ of quantum $\mathfrak {gl}(N)$ in ~\cite{BKLW13} by mimicking the approach in ~\cite{BLM90}.
Moreover, the  canonical bases of these modified coideal subalgebras are constructed and studied  for the first time.  
 Along the way,
a duality of Bao-Wang in ~\cite{BW}  relating  the (modified) coideal subalgebras and the Iwahori-Hecke algebra
 of type $\mbf B/ \mbf C$  associated to
$\mrm{SO}(2d+1)/\mrm{SP}(2d)$ is also geometrically realized  in a similar manner as the  type-$\mbf A$ case.
(See also \cite{G97} for a duality closely related to the  duality of Bao-Wang.)
The canonical basis theory for  these coideal subalgebras is initiated in the seminal work ~\cite{BW}, 
and is used substantially to give simultaneously a new formulation of  the Kazhdan-Lusztig conjecture of type $\mbf B/\mbf C$ on the irreducible character problem and the  resolution of  the analogous problem for the ortho-symplectic Lie superalgebras.

To this end, it is compelling to ask what happens to the remaining classical case: $G=\mrm{SO}(2d)$ of type $\mbf D$.
The purpose of this paper is to provide an answer to this question, as a sequel to ~\cite{BKLW13}.
More precisely, we obtain   two quantum algebras $\mcal K$ and $\mcal K^m$ via  the  geometry of
the $N$-step isotropic flag variety of type $\mbf D$ and a stabilization process following ~\cite{BLM90} and ~\cite{BKLW13}.
We show that both algebras  possess three distinguished bases, i.e., the standard, monomial
and canonical bases,  similar to the results in type $\mbf{ABC}$.
We further establish new dualities between these two algebras and the Iwahori-Hecke algebra of type $\mbf D$ attached to $\mrm{SO}(2d)$ algebraically and geometrically.

Unlike type $\mbf{ABC}$, the algebras $\mcal K$ and $\mcal K^m$ are not  modified forms of some known quantum algebras in literature, even though
they resemble the modified forms $\dot{\U}^{\imath}$, $\dot{\U}^{\jmath}$ of coideal subalgebras of quantum $\mathfrak{gl}(N)$.
It is natural to ask for a presentation of the two algebras by generators and relations. 
We have a complete answer for the algebra $\mcal K^m$, and partial results for $\mcal K$. 
We show that the algebra $\mcal K^m$ admits  defining relations similar to those of $\dot{\U}^{\imath}$, but with the size of the set of idempotent generators doubled,
after extending the underlying ring  to
the field  of rational functions.  
Despite all the similarities, we caution the reader that   $\dot{\U}^{\imath}$  is not a subalgebra of $\mcal K^m$.
The presentation   for $\mathcal K^m$ is obtained by showing that (the complexification of) $\mcal K^m$   is isomorphic to the modified form of
a new unital associative  algebra $\U^m$ containing  the coideal subalgebra $\U^{\imath}$ and two additional  idempotents. The appearance of the new idempotents 
reflects  the geometric fact that there are two connected components for maximal isotropic Grassmannians in the  type $\mbf D$ geometry. 
As for the bigger algebra $ \mcal K$,  we formulate another new unital associative  algebra $\U$ containing   the coideal subalgebra $\U^{\jmath}$ and three extra idempotents,
and expect its modified form to be isomorphic to $\mcal K$  after a suitable field extension.
As an  evidence in support of this expectation, we show that $\U$ and  the Iwahori-Hecke algebra of type $\mbf D$
satisfy a double centralizer property.
Notice  that the commuting actions between $\U^{\imath}$, $\U^{\jmath}$
and  the Iwahori-Hecke algebra of type $\mbf D$ are first observed in ~\cite[7.8]{ES13b} (see also ~\cite{ES13a}), so this result can be thought of as an enhancement of those in {\it loc. cit.}

As an application, we expect that  the type-$\mbf D$ duality and the canonical basis theory for the new algebras $\mcal K$ and $\mcal K^m$ developed in this paper will   shed light on  
the type-$\mbf D$ problems similar to those addressed in  ~\cite{BW}, currently under investigation by H. Bao.

Since our results are governed in principle  by  the (parabolic) Kazhdan-Lusztig polynomials of type $\mbf D$, they are obviously different from those in type $\mbf {ABC}$ in 
~\cite{BLM90} and ~\cite{BKLW13}.
Furthermore, the geometry of type $\mbf D$ is more challenging to handle.
In particular, there  are mainly three new technical barriers  in our type $\mbf D$  setting that we overcomed.
The first one is that there are two connected components for the maximal isotropic Grassmannian associated to $\mrm{SO}(2d)$. 
This forces us to parameterize the $\mrm{SO}(2d)$-orbits by using $signed$ $matrices$ instead of matrices in type $\mbf{ABC}$.
The second one is  that the number of isotropic lines in a given quadratic space of even dimension  over a finite field  depends on its isometric class, 
we choose  to work in the case when the group is split, which has a remarkable hereditary property (see Lemma ~\ref{lem1}).
The last one is that during  the stabilization process, one can not subtract/add  by an (even) multiple of the identity matrix as in
{\it loc. cit.}, because the signs of the matrices may change. To circumvent this difficulty, we subtract/add  an even multiple of
a matrix obtained by changing the middle entry to be zero in the diagonal of  the identity matrix.
All these factors make  the computations and arguments more involved than those in previous
cases.

As  this paper provides a complete picture  for the cases of the classical groups,
the problem of whether a similar picture exists for exceptional groups is still wide open.
Meanwhile, for $G$ replaced by a loop group of type $\mbf A$, 
there exists a similar geometric theory involving affine Iwahori-Hecke algebras of type $\mbf A$ and affine quantum $\mathfrak{gl}(N)$ in ~\cite{L99}, ~\cite{GV93}, ~\cite{SV00} and ~\cite{M10}.
The investigation for $G$ being  a loop group of type $\mbf{BCD}$ will be presented in a separate article. 

{\bf Acknowledgement.}
Y. Li thanks Huanchen Bao, Jonathan  Kujawa and Weiqiang Wang for fruitful collaborations, which pave the way for the current project.
We  thank Weiqiang Wang for comments on an earlier version of this article. 
Y. Li is partially supported by the NSF grant DMS 1160351.


\section{Schur dualities of type $\mbf D_d$}

In this section, we shall introduce the algebras $\U$ and $\U^m$, and formulate algebraically
the dualities between these two algebras and the  Iwahori-Hecke algebras of type $\mbf D_{d}$.

\subsection{The algebra $\U$ and the first duality}
\label{U-first}
Let $[a, b]$ denote the set of integers between $a$ and $b$.
Let $\U$ be the unital associative $\mbb Q(v)$-algebra generated by the symbols
$$
E_i, F_i, H_a^{\pm 1},  \ \mbox{and} \  J_{\alpha}, \quad
\forall i\in [1,n], a\in [1, n+1],  \alpha \in \{ +,0,  -\},
$$
satisfying the following relations.
\allowdisplaybreaks
\begin{eqnarray*}
&& J_++J_0+J_-=1,\quad J_{\alpha}J_{\beta}=\delta_{\alpha,\beta}J_{\alpha}, \\
&&  J_{\alpha}H_a=H_aJ_{\alpha},\\
&&J_{\pm}E_i=(1-\delta_{in})E_iJ_{\pm},\\
&& F_iJ_{\pm}= (1-\delta_{in})J_{\pm}F_i,\\
&&J_{\pm}F_nE_n-F_nE_nJ_{\mp}=\frac{H_n^{-1}H_{n+1}-H_nH_{n+1}^{-1}}{v-v^{-1}}(J_{\pm}-J_{\mp}),\\
&&H_aH_b=H_b H_a,\quad H_aH_a^{-1}=1,\\
&&H_aE_i=v^{ \delta_{a i} -\delta_{a,i+1}  - \delta_{a,n+1} \delta_{i,n}}E_iH_a, \\
&&H_aF_i=v^{- \delta_{ai} + \delta_{a,i+1} + \delta_{a,n+1}\delta_{i,n}}F_iH_a,\\
&& E_iF_j-F_jE_i=\delta_{ij}\frac{H_iH_{i+1}^{-1}-H_i^{-1}H_{i+1}}{v-v^{-1}},  \quad \quad \,   {\rm if} \ (i, j)\neq (n,n),\\
&&E_iE_j=E_jE_i,\quad F_iF_j=F_jF_i,\quad \hspace{2.1cm}  {\rm if}\ |i-j|>1,\\
&&E_i^2E_j-(v+v^{-1})E_iE_jE_i+E_jE_i^2=0,\quad \hspace{.51cm}  {\rm if}\ |i-j|=1,\\
&&F_i^2F_j-(v+v^{-1})F_iF_jF_i+F_jF_i^2=0,\quad \hspace{.71cm}  {\rm if}\ |i-j|=1,\\
&&E^2_nF_n+F_nE_n^2=(v+v^{-1})(E_nF_nE_n-E_n(vH_nH_{n+1}^{-1}+v^{-1}H_n^{-1}H_{n+1})),\\
&& F_n^2E_n+E_nF_n^2=(v+v^{-1})(F_nE_nF_n-(vH_nH_{n+1}^{-1}+v^{-1}H_n^{-1}H_{n+1})F_n),
\end{eqnarray*}
where $i, j\in [1,n]$, $a, b\in [1,n+1]$ and $\alpha\in \{ +, 0, -\}$.
Notice that the subalgebra $\mbf U^{\jmath}$ generated by $E_i$, $F_i$ and $H_a^{\pm 1}$ for any $i\in [1, n]$, $a\in[1,n+1]$ is the coideal subalgebra in the same notation in ~\cite{BKLW13}. See also ~\cite{Le02} and ~\cite{ES13b}.

Let $\mbf V$ be a vector space over $\mbb Q(v)$ of dimension $2n+1$.
We fix a basis $( \mbf v_i)_{ 1\leq i\leq 2n+1}$ for $\mbf V$.
Let $\mbf V^{\otimes d}$ be the $d$-th tensor space of $\mbf V$.
Thus we have a basis
$(\mbf v_{r_1}\otimes \cdots\otimes \mbf v_{r_d})$,  where $r_1,\cdots, r_d\in [1, 2n+1]$,  for the tensor space $\mbf V^{\otimes d}$.

For a sequence $\mbf r=(r_1,\cdots, r_d)$, we write $\mbf v_{\mbf r}$ for $\mbf v_{r_1}\otimes \cdots\otimes \mbf v_{r_d}$.
For a sequence $\mbf r=(r_1,\cdots, r_d)$, it defines uniquely a sequence of length $2d$ of the form
$$\tilde{ \mbf r}=(r_1, \cdots, r_d, 2n+2-r_d, 2n+2-r_{d-1},\cdots, 2n+2-r_1)$$
such that $r_i+r_{2d+1-i} =2n+2$.
We shall  identify $\mbf r$ with $\tilde{\mbf r}$ in what follows.

For a sequence $\mbf r$ and a fixed integer $p\in [1, 2d]$, we define  the sequence $\mbf r'_p$ and $\mbf r''_p$ by
\[
(\mbf r'_p)_j=
\begin{cases}
r_j, & j\neq p, 2d+1-p,\\
r_p +1, & j=p,\\
r_{2d+1 -p}-1, &j=2d+1-p,
\end{cases}
\quad
\mbox{and}
\quad
(\mbf r''_p)_j=
\begin{cases}
r_j, & j\neq p, 2d+1-p,\\
r_p -1, & j=p,\\
r_{2d+1 -p}+1, &j=2d+1-p.
\end{cases}
\]

\begin{lem}
\label{U-action-1}
We have  a left  $\U$-action on $\mbf V^{\otimes d}$ defined by, for any $i\in [1, n]$, $a\in [1, n+1]$,
 \begin{align}
 \label{U-action-1-formula}
E_i\cdot \mbf v_{\mbf r} &=v^{-\sum_{1\leq j\leq 2d} \delta_{i+1,r_j}}
\sum_{1\leq p\leq 2d: r_p=i}v^{2\sum_{j<p}\delta_{i+1,r_j}} \ \mbf v_{\mbf r'_p}, \nonumber\\
F_i \cdot \mbf v_{\mbf r}  &=v^{-\sum_{1\leq j\leq 2d} \delta_{i,r_j}}
\sum_{1\leq p\leq 2d:r_p=i+1}v^{2\sum_{j>p}\delta_{i, r_j}} \mbf v_{\mbf r''_p},  \nonumber \\
H_a^{\pm 1} \cdot \mbf{v_r}  &=v^{\mp \sum_{1\leq j\leq 2d} \delta_{a, r_j}  } \mbf{v_r}, \nonumber\\
J_{+}\cdot \mbf v_{\mbf r} &=
\begin{cases}
\mbf v_{\mbf r}, & r_i\neq n+1,\forall i, \#\{ j\in [1,d]| r_j\geq n+1\}  \ is \ even,\\
0, &\mbox{otherwise},
\end{cases}\nonumber \\
J_{-}\cdot \mbf v_{\mbf r} &=
\begin{cases}
\mbf v_{\mbf r}, & r_i\neq n+1,\forall i, \#\{ j\in [1,d]| r_j\geq n+1\}  \ is \ odd,\\
0, &\mbox{otherwise},
\end{cases}
\nonumber\\
J_{0}\cdot \mbf v_{\mbf r} &=
\begin{cases}
\mbf v_{\mbf r}, & r_i= n+1,\ \mbox{for some} \ i,\\
0, &\mbox{otherwise}.
\end{cases}
\nonumber
\end{align}
\end{lem}

The lemma  follows from  (\ref{V-V}),
Proposition ~\ref{prop3},  and Corollary ~\ref{cor9}.

Recall that the Iwahori-Hecke algebra $\Hd$ of type $\mbf D_d$ is a unital associative algebra over $\mathbb Q(v)$
generated by $\T_i$ for $i\in [1, d]$ and subject to the following relations.
\begin{equation*}
\begin{split}
& \T_i^2   =(v^2-1) \T_i + v^2 , \quad 1\leq i \leq d,\\
& \T_j \T_{j+1} \T_j  = \T_{j+1} \T_j \T_{j+1},\; \,  1\leq j\leq d-2,\\
& \T_i \T_j  =\T_j \T_i,\ \hspace{1.9cm} 1 \leq j\leq i-2\leq d-3,\\
& \T_d \T_l  = \T_l \T_d,\; \hspace{2cm}   l\neq d-2.\\
& \T_d \T_{d-2} \T_d  = \T_{d-2} \T_d \T_{d-2}.
\end{split}
\end{equation*}

\begin{lem}
\label{H-action-1}
We have a right $\Hd$-action on $\mbf V^{\otimes d}$ given by, for $1\leq j\leq d-1$,
\begin{eqnarray}
&\mbf v_{r_1 \dots r_{2d}} \T_j =
\begin{cases}
\mbf v_{r_1\dots r_{j-1} r_{j+1} r_j\dots r_{2d-j} r_{2d-j-1} r_{2d-j+1} \dots r_{2d}} ,  &  r_j < r_{j+1};\\
v^2 \mbf v_{r_1\dots r_{2d}}, &  r_j = r_{j+1};\\
(v^2-1) \mbf v_{r_1\dots r_{2d}} + v^2 \mbf v_{r_1 \dots r_{j-1} r_{j+1} r_j  \dots  r_{2d-j} r_{2d-j-1} r_{2d-j+1} \dots r_{2d}}, & r_j> r_{j+1}.
\end{cases}
\vspace{5pt}\\
&\hspace{12pt} \mbf v_{r_1 \dots r_{2d}} \T_d =
\begin{cases}
\mbf v_{r_1\dots r_{d-2} r_{d+1} r_{d+2}r_{d-1}r_d r_{d+3}  \dots r_{2d}} ,  & r_{d-1}+r_d < N+1;\\
v^2 \mbf v_{r_1\dots r_{2d}}, &  r_{d-1}+r_d =N+1;\\
(v^2-1) \mbf v_{r_1\dots r_{2d}} + v^2 \mbf v_{r_1\dots r_{d-2} r_{d+1} r_{d+2}r_{d-1}r_d r_{d+3}  \dots r_{2d}}, &r_{d-1}+r_d > N+1.
\end{cases}
\end{eqnarray}
Here we identify the sequence $\mbf r$ with the associated sequence $\tilde{ \mbf r}$.
\end{lem}

This lemma follows from (\ref{V-V}) and  Lemmas ~\ref{Geometric-Hecke} and  ~\ref{H-action-Y}.

We now can state the first duality.

\begin{prop}
The left $\U$-action in Lemma ~\ref{U-action-1} and the right $\Hd$-action in Lemma ~\ref{H-action-1}  on $\mbf V^{\otimes d}$ are commuting.  They form a double centralizer for $n\geq d$, i.e.,
$$
\Hd \simeq \End_{\U} (\mbf V^{\otimes d})
\quad
\mbox{and}
\quad
\U\to \End_{\Hd} (\mbf V^{\otimes d}) \ \mbox{is surjective}.
$$
\end{prop}

The proposition follows from  the previous two lemmas, Lemma ~\ref{eq34},  Proposition ~\ref{prop3} and Corollary ~\ref{S-generator}.


\subsection{The algebra $\U^m$ and the second duality}
\label{Um}

Let $\U^m$ be an associative $\mbb Q(v)$-algebra with unit generated by the symbols
\[
E_i, F_i, H_a^{\pm 1}, T, J_{\alpha}, \quad \forall i\in [1, n-1], a\in [1, n], \alpha\in \{ +, -\},
\]
and subject to the following defining relations.
  \allowdisplaybreaks
\begin{eqnarray*}
&& J_++J_-=1,\quad J_{\alpha}J_{\beta}=\delta_{\alpha,\beta}J_{\alpha}, \\
&&  J_{\pm}H_a=H_aJ_{\pm},\quad J_{\pm}E_i=E_iJ_{\pm},\quad F_iJ_{\pm}= J_{\pm}F_i,\\
&&J_{\pm}T = TJ_{\mp},\\
&&H_aH_b=H_b H_a,\quad H_a H_a^{-1}=1,\\
&&H_aE_i=v^{-\delta_{a,i+1} + \delta_{a i} - \delta_{a,n+1}\delta_{i,n}}E_iH_a,\quad \\
&&H_aF_i=v^{- \delta_{ai}  + \delta_{a,i+1} +  \delta_{a,n+1}\delta_{i,n}}F_iH_a,\\
&&H_i T=TH_i,\\
&& E_iF_j-F_jE_i=\delta_{ij}\frac{H_iH_{i+1}^{-1}-H_i^{-1}H_{i+1}}{v-v^{-1}},\\
&&E_iE_j=E_jE_i,\quad F_iF_j=F_jF_i,\quad \hspace{2cm}  {\rm if}\ |i-j|>1,\\
&&E_i^2E_j-(v+v^{-1})E_iE_jE_i+E_jE_i^2=0,\quad  \quad {\rm if}  \ |i-j|=1,\\
&&F_i^2F_j-(v+v^{-1})F_iF_jF_i+F_jF_i^2=0,\quad \hspace{.6cm} {\rm if}\ |i-j|=1,\\
&&TE_i=E_iT,\ TF_i=F_iT,\quad \hspace{2.87cm}  {\rm if} \  i\leq n-2,\\
&&E_{n-1}^2T-(v+v^{-1}) E_{n-1}TE_{n-1}+TE_{n-1}^2=0,\\
&&F_{n-1}^2T- (v+v^{-1} F_{n-1}TF_{n-1}+TF_{n-1}^2=0,\\
&&T^2E_{n-1}-(v+v^{-1})TE_{n-1}T+E_{n-1}T^2=E_{n-1},\\
&&T^2F_{n-1}-(v+v^{-1})TF_{n-1}T+F_{n-1}T^2=F_{n-1}.
\end{eqnarray*}

Note that the subalgebra $\U^{\imath}$  generated by $E_i$, $F_i$, $H_a^{\pm 1}$ and $T$ is  the algebra in the same notation in ~\cite[5.3]{BKLW13}. See also ~\cite{Le02} and ~\cite{ES13b}.

Let $\mbf W$ be the subspace of $\mbf V$ spanned by the basis elements $\mbf v_i$ for $i\neq n+1$.
Its $d$-th tensor space  $\mbf W^{\otimes d}$ is naturally a subspace of $\mbf V^{\otimes d}$ spanned by the vectors $\mbf{v_r}$ such that $r_i\neq n+1$ for any $i$.
Then we see that $\mbf W^{\otimes d}$ is a vector space of dimension $(2n)^d$.

By Lemma ~\ref{H-action-1}, the $\Hd$-action on $\mbf V^{\otimes d}$ induces an $\Hd$-action on   $\mbf W^{\otimes d}$.
Moreover,

\begin{lem}
\label{U-action-2}
We have a $\U^m$-action on $\mbf W^{\otimes d}$ given by
the same formulae for $E_i$, $F_i$, $H_a^{\pm 1}$ and $J_{\pm}$ for $i\in [1, n-1]$ and $a\in [1, n]$ as in Lemma ~\ref{U-action-1}, together with
\[
T\cdot\mbf{v_r}= \left (F_nE_n + \frac{H_n H_{n+1}^{-1} - H_n^{-1} H_{n+1}}{v-v^{-1}} \right )\cdot \mbf{v_r}.
\]
\end{lem}

This lemma  is proved by (\ref{W-W}),  (\ref{T=FE}), Lemma ~\ref{Sm-W} and  Proposition \ref{prop7}.
We can now state the second duality.

\begin{prop}
The $\U^m$-action and $\Hd$-action on $\mbf W^{\otimes d}$ are commuting. They  enjoy the double centralizer property when $n\geq d$.
\end{prop}

The proof is given by Lemma ~\ref{U-action-2}, Lemma ~\ref{Geometric-duality-II},  Proposition ~\ref{prop7} and Proposition ~\ref{lumpsum}.

\section{A geometric setting}

We now turn to the  geometric setting in order to prove the above results among others.

\subsection{Preliminary}

We start by recalling some results on counting isotropic subspaces in an even dimensional quadratic  space over a finite field.  We refer to ~\cite{W93} and the references therein for more details.

Let $\mbb F_q$ be a finite field of $q$ elements and of odd characteristic.
Recall that $d$ is a fixed positive integer, and  we set $$D=2d.$$
On the $D$-dimensional vector space $\mbb F_q^D$, we  fix a non-degenerate symmetric bilinear form $Q$ whose associated  matrix  is
\begin{align}
\label{bilinear-form}
\begin{bmatrix}
  0 &I_d\\
  I_d & 0
\end{bmatrix}
\end{align}
under the standard  basis of $\mbb F_q^D$.
By convention,  $W^{\perp}$ stands  for  the orthogonal complement of a vector subspace  $W$ in $\mbb F_q^D$.
Moreover, we call $W$ isotropic if $W\subseteq W^{\perp}$. We write $|W| $ for the dimension of $W$.

For any isotropic subspace $W$, the bilinear form $Q$ induces a non-degenerate symmetric bilinear form $Q|_{W^{\perp}/W}$ on $W^{\perp}/W$. One of the reasons that we fix $Q$ of the form (\ref{bilinear-form}) is its hereditary property in the following lemma, which can be proved inductively.

\begin{lem}\label{lem1}
The associated matrix  of  $Q|_{W^{\perp}/W}$ is of the form (\ref{bilinear-form}) with rank $d-2|W|$
under a certain basis.
\end{lem}

By using Lemma ~\ref{lem1}, we can  count the number of isotropic lines inductively to get the following lemma.

\begin{lem}
\label{S_d}
The cardinality, $S_d$, of the set of isotropic lines in $\mbb F_q^D$ is $\frac{q^{2d-1}-1}{q-1}+q^{d-1}$.
\end{lem}

We remark that the order of the set  $S_d$  with respect to a symmetric bilinear form on $\mbb F_q^D$  not isometric to $Q$ may not be the same as the number in the above lemma.
We will need the following lemma later.
We write $W\overset{a}{\subset} V$ if $W\subseteq V$ and $|V/W|=a$.

\begin{lem}
\label{countingisotrop}
Let $V=(V_i)_{1\leq i\leq 5}$ be a fixed flag of $\mbb F_q^D$
such that $V_{i-1} \subset V_{i}$, $V_i=V_{5-i}^{\bot}$, $|V_i/V_{i-1}|=a_i$  and $a_i \geq 0$, for any  $i\in [1,5]$.
Consider the sets
$$Z_i=\{U\subset V_i| |U|=1, U\ \text{is isotropic}, \ U\not \subseteq V_{i-1}\}, \quad \forall i\in [1, 5].$$
We have
$$(\mrm i). \
 \# Z_3=q^{a_1+a_2} \left (\frac{q^{a_3-1}-1}{q-1}+q^{\frac{a_3}{2}-1} \right )
\quad \mbox{and} \quad
(\mrm{ii}). \
 \# Z_4=q^{a_1+a_2+a_3-1}\frac{q^{a_4}-1}{q-1}.$$
\end{lem}

\begin{proof}
To prove (i), we consider the set $Z_3'=\{W\subset V_3| V_2 \overset{1}{\subset} W, W\ \text{is isotropic}\}$.
Let $\phi: Z_3\rightarrow Z'_3$ be the map defined by  $U\mapsto V_2+U$. Clearly, the map $\phi$ is surjective.
Observe that  the order of each fiber is $q^{|V_2|}=q^{a_1+a_2}$ and, moreover,
$Z_3'$ gets identified with the set of all isotropic lines in $V_3/V_2$.
By Lemma ~\ref{S_d}, we have
$\#Z_3'=\frac{q^{a_3-1}-1}{q-1}+q^{\frac{a_3}{2}-1}.$
This proves (i).

We now prove (ii). Consider the set $Z_4'=\{W\subset V_4| V_1 \overset{1}{\subset} W, W\not \subset V_3, W\ \text{is isotropic}\}$.
Consider the map  $\phi': Z_4\rightarrow Z'_4$, $U\mapsto V_1+U$.
This is a well-defined and surjective  map and the cardinality of each fiber is $q^{|V_1|}=q^{a_1}$.
To calculate the cardinality of $Z_4'$,
we set $$\tilde Z_i=\{U\subset V_i/V_1||U|=1, U\ \text{is isotropic}\}.$$
By a similar calculation in (i), we have
\begin{align*}
\#Z'_4
&=\#\tilde Z_4-\#\tilde Z_3=\#\tilde Z_4-(\#\tilde Z_3\setminus \tilde Z_2+\#\tilde Z_2)\\
&=\frac{q^{a_2+a_3+a_4-1}-1}{q-1}+q^{a_2+\frac{a_3}{2}-1}-
\left (q^{a_2} (\frac{q^{a_3-1}-1}{q-1}+q^{\frac{a_3}{2}-1})+\frac{q^{a_2}-1}{q-1}\right )\\
&=q^{a_2+a_3-1}\frac{q^{a_4}-1}{q-1}.
\end{align*}
This proves (ii).
\end{proof}

\subsection{The first double centralizer} 
\label{setup}

We fix another  positive integer $n$ and let
\(
N=2n+1.
\)
We fix a maximal isotropic vector subspace $M_d$ in $\mbb F_q^D$ (of dimension $d$).  Consider the following sets.
\begin{itemize}
\item The set $\mscr X$ of $N$-step flags $V=( V_i)_{ 0\leq i\leq N}$
          in $\mbb F_q^D$ such that  $V_i\subseteq V_{i+1}$, $V_i= V_j^{\perp}$, for any   $i+j = N$.

\item The set $\mscr Y$  of complete flags  $F= (F_i)_{0\leq i\leq D} $ in $\mbb F_q^D$
          such that $F_i\subset F_{i+1}$, $|F_i|=i$ and  $ F_i = F_j^{\perp}$, for any $i+ j = D$,  and $|F_d\cap M_d|\equiv d \bmod 2$.
\end{itemize}

Let $G=\mrm{SO}(D)$ be the special orthogonal group attached to $Q$.
The sets $\mscr X$ and $\mscr Y$ admit naturally $G$-action from the left. Moreoever,
$G$ acts  transitively on $\mscr Y$ thanks to
the condition $|F_d\cap M_d|\equiv d \bmod 2$.
Let $G$ act diagonally on the product $\mscr X\times \mscr X$ (resp. $\mscr X\times \mscr Y$ and $\mscr Y\times \mscr Y$).
Set
\begin{equation}\label{eq32}
  \mcal A=\mbb Z[v, v^{-1}].
\end{equation}
Let
\begin{align}
\label{S(X)}
\mcal S_{\mscr X}=\mcal A_G(\mscr X\times\mscr X)
\end{align}
be the set of all $\mcal A$-valued $G$-invariant functions on $\mscr X\times \mscr X$.
Clearly, the set $\mcal S_{\mscr X}$ is a  free $\mcal A$-module.
Moreover,  $\mcal S_{\mscr X}$ admits  an associative  $\mcal A$-algebra structure `$*$' under a standard  convolution product as discussed in ~\cite[2.3]{BKLW13}. In particular, when $v$ is specialized to $\sqrt q$, we have
\begin{equation}
  \label{eq30}
  f * g(V, V')=\sum_{V''\in \mscr X}f(V, V'')g(V'',V'), \quad \forall\ V,V'\in \mscr X.
\end{equation}
Similarly, we define  the free $\mcal A$-modules
\begin{equation}\label{V}
\mcal V=\mcal A_G(\mscr X\times\mscr Y)
\quad
\mbox{and}
\quad
\mcal H_{\mscr Y}=\mcal A_G(\mscr Y\times \mscr Y).
\end{equation}
A similar  convolution product gives an associative algebra structure on $\mcal H_{\mscr Y}$ and
 a left $\mcal S_{\mscr X}$-action   and a right $\mcal H_{\mscr Y}$-action on $\mcal V$.
Moreover, these two actions commute and hence we have the following $\mcal A$-algebra homomorphisms.
\[
\mcal S_{\mscr X} \to \End_{\mcal H_{\mscr Y}} (\mcal V)
\quad\mbox{and}\quad
\mcal H_{\mscr Y} \to \End_{\mcal S_{\mscr X}} (\mcal V).
\]
By ~\cite[Theorem 2.1]{P09}, we have the following double centralizer property.

\begin{lem}
\label{eq34}
$\End_{\mcal H_{\mscr Y}}(\mcal V)\simeq \mcal S_{\mscr X}$ and
$\End_{\mcal S_{\mscr X}}(\mcal V) \simeq \mcal H_{\mscr Y}$, if
$ n\geq d.$
\end{lem}

We note that the  result in ~\cite[Theorem 2.1]{P09} is obtained over  the field $\mbb C$ of complex numbers,
but the  proof can be adapted  to  our setting   over the ring  $\mcal A$.

\subsection{$G$-orbits on $\mscr X\times \mscr Y$ and $\mscr Y\times \mscr Y$}
\label{sec3.2}

We shall give a description of the $G$-orbits on $\mscr X\times \mscr Y$  and  $\mscr Y\times \mscr Y$.
The description of the $G$-orbits on $\mscr X\times \mscr X$ is  more complicated, and  postponed until   Section ~\ref{sec4.2}.

We start by introducing the following notations associated to a matrix
$M=(m_{ij})_{1\leq i, j \leq c}$.
\begin{align} \label{ro-co}
\begin{split}
\ro (M) & = \left (\sum_{j=1}^{c}  m_{ij} \right )_{1\leq i\leq c}, \\
\co (M) & = \left (\sum_{i=1}^{c} m_{ij} \right )_{1\leq j \leq c}, \\
\ur (M)  & = \sum_{i\leq c/2, j> c/2} m_{ij}.
\end{split}
\end{align}
We also write $\ro(M)_i$ and $\co (M)_j$ for the $i$-th and $j$-th component of the row vectors of $\ro (M)$ and $\co(M)$, respectively.

To a pair $(F, F') \in \mscr Y\times \mscr Y$, we can associate a $D\times D$ matrix $\sigma=(\sigma_{ij})$ by setting
\begin{equation}\label{bi}
\sigma_{ij} = \dim \frac{F_{i-1}+ F_i \cap F_j'}{F_{i-1} +F_i \cap F_{j-1}'}, \quad \forall 1\leq i, j\leq D.
\end{equation}
This assignment defines a bijection
\begin{align}
G \backslash \mscr Y\times \mscr Y \simeq  \Sigma,
\end{align}
where
$\Sigma$ is the set of all matrices $\sigma \equiv (\sigma_{ij})$ in $\mbox{Mat}_{D\times D} (\mbb N)$ such that
\begin{align*}
 \ro (\sigma)_i = 1, \quad  \ro (\sigma)_j =1, \quad \sigma_{ij} = \sigma_{D+1-i, D+1-j},\quad \ur (\sigma) \equiv 0 \bmod 2,\quad\forall i, j\in [1, D].
\end{align*}

A similar assignment yields a bijection
\begin{align}\label{eq44}
G \backslash \mscr X\times \mscr Y \simeq \Pi ,
\end{align}
where the set $\Pi$ consists of all matrices $B=(b_{ij})$ in $\mrm{Mat}_{N\times D} (\mbb N)$ subject to
\[
 \co (B)_j =1;\quad  b_{ij} = b_{N+1- i, D+1-j},\  \forall\ i\in [1, N],\  j\in [1, D].
 \]
Moreover, we have
\begin{equation}
 \begin{split}
\# \Sigma = 2^{d-1} \cdot d! \quad {\rm and}\quad
\# \Pi = (2n+1)^d.
\end{split}
\end{equation}

\subsection{$\mcal H_{\mscr Y}$-action on $\mcal V$}
We shall provide an explicit description of the action of $\mcal H_{\mscr Y}$ on $\mcal V$.
For any $1\leq j\leq d-1$, we define  a function $\T_j$ in $\mcal H_{\mscr Y}$  by
\begin{align*}
\T_j (F, F') &=
\begin{cases}
1, & \mbox{if} \; F_i = F_i' \ \forall i\in [1, d]\backslash  \{j\}, F_j\neq F_j';\\
0, & \mbox{otherwise}.
\end{cases}\\
\T_d &= e_{(d-1, d+1)(d, d+2)},
\end{align*}
where $e_{(d-1, d+1)(d, d+2)}$ is
the  characteristic function of the $G$-orbit corresponding to the permutation  matrix $(d-1, d+1)(d, d+2)$.
Then we have the following well-known result.

\begin{lem} \label{Geometric-Hecke}
 The assignment of sending the  functions $\T_j$,  for $1\leq j\leq  d$, in  the algebra $\mcal H_{\mscr Y}$ to the generators of
 $\Hd$ in the same notations is an isomorphism.
\end{lem}

Given $B=(b_{ij})\in \Pi$, let $r_c$ be the unique number in $[1,N]$ such that $b_{r_c, c} =1$ for each $c\in [1, D]$.
The correspondence $B\mapsto \tilde{\mbf r}= (r_1,\cdots, r_D)$ defines a bijection between $\Pi$ and the set of all sequences
$(r_1,\cdots, r_d)$ such that $r_i+ r_{D+1-i} = N+1$ for any $i\in [1, D]$.
Denote by $e_{r_1 \dots r_D}$ the characteristic function of the $G$-orbit corresponding to the matrix $B$ in $\mcal V$.
It is clear that the collection of these characteristic functions provides a basis for $\mcal V$.

Recall from Section ~\ref{U-first} that we have the space $\mbf V^{\otimes d}$ spanned by vectors $\mbf{v_r}$ and
to each sequence $\mbf r$ a sequence $\tilde{\mbf r}$ is uniquely defined.
Thus we have an isomorphism of vector spaces over $\mathbb Q(v)$:
\begin{align}\label{V-V}
\mbf V^{\otimes d} \to \mbb Q(v)\otimes_{\mcal A}  \mcal V, \quad   \mbf v_{\mbf r} \mapsto e_{r_1,\cdots, r_D}.
\end{align}
Moreover, we have

\begin{lem}
\label{H-action-Y}
The action of $\mcal H_{\mscr Y}$ on $\mcal V$ is described as follows.
For $1\leq j\leq d-1$, we have
\begin{eqnarray}
&e_{r_1 \dots r_D} \T_j =
\begin{cases}
e_{r_1\dots r_{j-1} r_{j+1} r_j\dots r_{D-j} r_{D-j-1} r_{D-j+1} \dots r_D} ,  &  r_j < r_{j+1};\\
v^2 e_{r_1\dots r_D}, &  r_j = r_{j+1};\\
(v^2-1) e_{r_1\dots r_D} + v^2 e_{r_1 \dots r_{j-1} r_{j+1} r_j  \dots  r_{D-j} r_{D-j-1} r_{D-j+1} \dots r_D}, & r_j> r_{j+1}.
\end{cases}
\label{1}
\vspace{5pt}\\
&\hspace{12pt}e_{r_1 \dots r_{D}} \T_d =
\begin{cases}
e_{r_1\dots r_{d-2} r_{d+1} r_{d+2}r_{d-1}r_d r_{d+3}  \dots r_D} ,  & r_{d-1}+r_d < N+1;\\
v^2 e_{r_1\dots r_D}, &  r_{d-1}+r_d =N+1;\\
(v^2-1) e_{r_1\dots r_D} + v^2 e_{r_1\dots r_{d-2} r_{d+1} r_{d+2}r_{d-1}r_d r_{d+3}  \dots r_D}, &r_{d-1}+r_d > N+1.
\end{cases}
\label{2}
\end{eqnarray}
\end{lem}

\begin{proof}
Formula (\ref{1}) agrees with the one in  ~\cite[1.12]{GL92}, whose  proof is also the same as the one for type-$\mbf A$ case.
We shall  prove (\ref{2}).  It suffices to show the result by specializing $v$ to $\sqrt q$.
By the  definition of convolution product, we have
\[
e_{r_1 \dots r_{D}} \T_d (V, F) =\sum_{F'\in \mscr Y} e_{r_1 \dots r_{D}} (V, F')  \T_d(F', F).
\]
By the definition of $\T_d$, we have $F'_i=F_i$ if $i\neq d-1,d, d+1$ or $d+2$.
So the calculation is reduced to the case when $D=4$.
Note also that it is enough to calculate the case when $n=2$, which we will assume.

If two of $r_1, r_2, r_3, r_4$ are equal, then the calculation can be  reduced further to the case when $n=1$.
In this case, we have
  \begin{eqnarray*}
e_{2132}\T_2=e_{3221},\quad & e_{2222} \T_2=v^2e_{2222},\quad & e_{2312}\T_2=(v^2-1)e_{2312}+v^2e_{1223}.\\
e_{1223}\T_2=e_{2312},\quad & e_{1313} \T_2=v^2e_{1313},\quad & e_{3221}\T_2=(v^2-1)e_{3221}+v^2e_{2132}.\\
e_{1133}\T_2=e_{3311},\quad & e_{3131} \T_2=v^2e_{3131},\quad & e_{3311}\T_2=(v^2-1)e_{3311}+v^2e_{1133}.
  \end{eqnarray*}
  For the case when  $r_1, r_2, r_3, r_4$ are all distinct, we have
  \begin{eqnarray*}
e_{2514}\T_2=(v^2-1)e_{2514}+v^2e_{1425},\quad && e_{4512} \T_2=(v^2-1)e_{4512}+v^2e_{1245}.\\
e_{5241}\T_2=(v^2-1)e_{5241}+v^2e_{4152},\quad && e_{5421} \T_2=(v^2-1)e_{5421}+v^2e_{2154}.\\
e_{1245}\T_2=e_{4512},\quad  e_{1425} \T_2=e_{2514},\quad && e_{2154} \T_2=e_{5421},\quad  e_{4152} \T_2=e_{5241}.
  \end{eqnarray*}
Formula (\ref{2}) follows from  the above computations.
\end{proof}

\section{Calculus of the algebra $\mcal S$}
\label{secS_D}

Recall from the previous section  that $\mcal S_{\mscr X}$ is the convolution algebra on $\mscr X\times \mscr X$
defined in (\ref{S(X)}).
For simplicity, we shall denote $\mcal S$ instead of $\mcal S_{\mscr X}$.
In this section, we determine  the generators for $\mcal S$ and the associated  multiplication formula.
Furthermore, we provide with   a (conjectural)  algebraic presentation of $\mcal S$  and deduce various bases.

\subsection{Defining relations of $\mcal S$}

 For any $i\in [1, n]$, $a\in [1, n+1]$, we set
\begin{align}
\begin{split}
E_i (V, V') &=
\begin{cases}
v^{-|V'_{i+1}/V'_i|}, &\mbox{if}\; V_i\overset{1}{\subset} V_i', V_j=V_{j'},\forall j\in [1,n]\backslash \{i\}; \\
0, &\mbox{otherwise}.
\end{cases}
\\
F_i (V, V') &=
\begin{cases}
v^{-|V'_i/V'_{i-1}|}, &\mbox{if}\; V_i\overset{1}{\supset} V_i', V_j=V_{j'},\forall j\in [1,n]\backslash \{i\}; \\
0, &\mbox{otherwise}.
\end{cases}
\\
H_a^{\pm 1} (V, V') & =
\begin{cases}
v^{\mp |V_a'/V_{a-1}'|} , &\mbox{if}\; V=V';\\
0, & \mbox{otherwise}.
\end{cases}
\\
J_+(V,V') &=
\begin{cases}
  1 , &\mbox{if}\; V=V',\ |V_n|=d\ {\rm and}\ |V_n\cap M_d|\equiv d\bmod 2;\\
0, & \mbox{otherwise}.
\end{cases}
\\
J_-(V,V') &=
\begin{cases}
  1 , &\mbox{if}\; V=V',\ |V_n|=d\ {\rm and}\ |V_n\cap M_d|\equiv d-1\bmod 2;\\
0, & \mbox{otherwise}.
\end{cases}
\\
J_0&=1-J_+-J_-.
\end{split}
\end{align}

It is clear that these functions are elements in $\mcal S$.

\begin{prop}\label{prop3}
The functions $E_i, F_i$,  $H_a^{\pm 1}$  and $J_{\alpha}$  in $\mcal S$, for any $i\in [1,n]$, $a\in [1,n+1]$ and $\alpha\in \{ \pm, 0\}$,   satisfy
the defining relations of the algebra $\U$ in Section  ~\ref{U-first}, together with the following ones.
\begin{equation}\label{relation B1}
 H_{n+1}\prod_{i=1}^n H_i\,^2=v^{-D},\quad {\rm and}\quad \prod_{l=1}^d(H_j-v^{\redtext{-}l})=0,\quad  \forall j\in [1,n].
\end{equation}
\end{prop}
\begin{proof}
The proofs of the  identities in the first four rows of the defining relations of $\U$ are straightforward.
We show the  identity in the fifth row.
Let $\lambda'_i=|V'_i/V'_{i-1}|$.
We have
\begin{equation*}
F_nE_{n}(V,V')=\left\{\begin{array}{ll}
 \frac{v^{2\lambda_n'}-1}{v^2-1}v^{-\lambda_n'-\lambda'_{n+1}+1}& {\rm if}\ V=V',\vspace{5pt}\\
 v^{-\lambda_n'-\lambda'_{n+1}+1}& {\rm if}\ |V_n\cap V_n'|=|V_n|-1=|V_n'|-1,\vspace{5pt}\\
  0&{\rm otherwise}.
  \end{array}
  \right.
\end{equation*}
We set
$$\mscr X^3=\{V\in \mscr X||V_n|=d,\ |V_n\cap M_d|\equiv d-1\bmod 2\}.$$
It  is  a $G$-orbit and we have $|V_n\cap V_n'|\equiv d\bmod 2$ for any $V, V'\in \mscr X^3$.
Therefore,
\begin{equation*}
(J_+F_nE_{n}-F_nE_nJ_-)(V,V')=\left\{\begin{array}{ll}
 \frac{v^{\lambda_n'}-v^{-\lambda_n'}}{v-v^{-1}}& {\rm if}\ V=V'\not\in \mscr X^3,\vspace{5pt}\\
 -\frac{v^{\lambda_n'}-v^{-\lambda_n'}}{v-v^{-1}}& {\rm if}\ V=V'\in \mscr X^3,\vspace{5pt}\\
  0&{\rm otherwise}.
  \end{array}\right.
\end{equation*}
It is easy to check that the right hand side is equal to $\frac{H_n^{-1}H_{n+1}-H_nH_{n+1}^{-1}}{v-v^{-1}}(J_{+}-J_{-})(V,V')$.

We now show the penultimate identity.
By a direct calculation, we have
  \begin{equation*}
    E_nF_nE_n(V,V')=\left\{\begin{array}{l}
      (\frac{v^{2\lambda'_n}-1}{v^2-1}+\frac{v^{2\lambda_{n+1}'+2}-1}{v^2-1}
       +v^{D-2|V_n|-2}-1)v^{-\lambda'_n-2\lambda_{n+1}'+1},
       \quad {\rm if}\ V_n \overset{1}{\subset} V'_n,
       \vspace{5pt}\\
      v^{-\lambda'_n-2\lambda_{n+1}'+1}, \hspace{52pt} {\rm if}\ |V_n \cap V'_n|=|V_n|-1
      \ {\rm and}\ V_n\not\subset V'_{n+1},\\
      (v^2+1)v^{-\lambda'_n-2\lambda_{n+1}'+1},\quad {\rm if}\ |V_n \cap V'_n|=|V_n|-1
      \ {\rm and}\ V_n \subset V'_{n+1},\\
      0,\hspace{110pt}{\rm otherwise}.
    \end{array}
  \right.
  \end{equation*}
  \begin{equation*}
    E_n^2F_n(V,V')=\left\{\begin{array}{l}
      (\frac{v^{2\lambda_{n+1}'-2}-1}{v^2-1}
       +v^{D-2|V'_n|-2})(v^2+1)v^{-\lambda'_n-2\lambda_{n+1}'+2},
       \quad {\rm if}\ V_n \overset{1}{\subset} V'_n\ {\rm and}\ |V'_n|<d,
       \vspace{5pt}\\
      (v^2+1)v^{-\lambda'_n-2\lambda_{n+1}'+2} ,\quad {\rm if}\ |V_n \cap V'_n|=|V_n|-1, |V'_n|<d
      \ {\rm and}\ V_n \subset V'_{n+1},\\
      0, \hspace{112pt} {\rm otherwise}.
    \end{array}
  \right.
  \end{equation*}
  \begin{equation*}
    F_n E_n^2(V,V')=\left\{\begin{array}{ll}
      \frac{v^{2\lambda_{n}'-2}-1}{v^2-1}
       (v^2+1)v^{-\lambda'_n-2\lambda_{n+1}'},
       &{\rm if}\ V_n \overset{1}{\subset} V'_n,\\
      (v^2+1)v^{-\lambda'_n-2\lambda_{n+1}'}, &{\rm if}\ |V_n \cap V'_n|=|V_n|-1,\hspace{90pt}\\
      0, & {\rm otherwise}.
    \end{array}
  \right.
\end{equation*}
The penultimate identity follows.

To prove the last identity,
we define a map $\rho: \mcal S \rightarrow \mcal S$ such that $\rho(f)(V,V')=f(V',V)$.
It is clear that  $\rho$ is an anti-automorphism.
Moreover, we have
\begin{align} \label{rho}
\rho(E_n)=v^{-1}H_n^{-1}H_{n+1}F_n,\quad  \rho(F_n)=v^{-2}H_nH_{n+1}^{-1} E_n,
\quad \mbox{and}\quad
 \rho(H_n^{\pm 1})=H_n^{\pm 1}.
\end{align}
Applying $\rho$ to both sides of the penultimate  identity, we get the last identity.
The rest relations are reduced to type-$\mbf A$ case, and will not be reproduced here.
\end{proof}

\subsection{Parametrization of $G$-orbits on $\mscr X\times \mscr X$}
\label{sec4.2}

In order to describe the structure of the algebra $\mcal S$, we need to parametrize the $G$-orbits in $\mscr X\times \mscr X$.
Recall from Section ~\ref{setup} that
$\mscr X$ is the set of $N$-step flags in $\mbb F_q^D$ such that $ V_i= V_j^{\perp}, \forall\  i+j = N$.
For any pair $(V, V')$ of flags in $\mscr X$, we can assign an $N$ by $N$ matrix as (\ref{bi}) whose $(i,j)$-entry equal to
$\dim \frac{V_{i-1}+ V_i\cap V_j'}{V_{i-1} + V_i\cap V_{j-1}'}$. It is clear that this assignment is $G$-invariant.
Thus we have a map
\begin{align}\label{tphi}
\tilde \Phi: G\backslash \mscr X\times \mscr X \to \Xi,
\end{align}
where  $\Xi$ is the set of all $N\times N$ matrices  $A$  with entries in $\mbb N$ subject to
\[
\sum_{i,j\in [1,N]}a_{ij}=D,\ a_{ij} = a_{N+1 -i, N+1-j}, \quad  \forall i, j\in [1, N].
\]
This map is surjective, but not injective. We need to refine it.

Recall that $M_d$ is a fixed maximal isotropic subspace in $\mbb F_q^D$ and
$$\mscr X^3=\{ V \in \mscr X| |V_n|=d\ {\rm and}\ |V_n\cap M_d| \equiv d-1\bmod 2\}.$$
We set
\begin{itemize}
  \item $\mscr X^1=\{V\in \mscr X | |V_n|<d\}$,
  \item $\mscr X^2=\{ V \in \mscr X| |V_n|=d\ {\rm and}\ |V_n\cap M_d| \equiv d\bmod 2\}$.
\end{itemize}
We have a partition of $\mscr X$:
\[
\mscr X =\mscr X^1\sqcup \mscr X^2 \sqcup \mscr X^3.
\]

Let $O(D)$ be the orthogonal group associated to $Q$.
For any $g\in O(D)\setminus G$, the
map $\psi_g: \mscr X^2 \rightarrow \mscr X^3$, defined by  $V\mapsto g\cdot V$,  is a bijection,
 which yields the following bijections.
\begin{equation}
\label{parity}
\begin{split}
& G \backslash \mscr X^1\times \mscr X^2 \rightarrow G \backslash \mscr X^1\times \mscr X^3, \quad
 G \backslash \mscr X^2\times \mscr X^1 \rightarrow G \backslash \mscr X^3\times \mscr X^1, \\
& G \backslash \mscr X^2\times \mscr X^2 \rightarrow G \backslash \mscr X^3\times \mscr X^3, \quad
 G \backslash \mscr X^2\times \mscr X^3 \rightarrow  G \backslash \mscr X^3\times \mscr X^2.
\end{split}
\end{equation}
Moreover, corresponding  pairs on both sides under  the bijections in (\ref{parity}) get sent to the same matrix by $\tilde \Phi$.
In corresponding to (\ref{parity}), we define a sign function
\begin{align}
\label{sign-e}
\sgn (i, j)  =
\begin{cases}
0 & (i, j) =(1, 1),\\
+ & (i, j)= (1,2), (2, 1), (2,2), (2,3),\\
- & (i, j)= (1, 3), (3, 1), (3,3), (3,2).
\end{cases}
\end{align}

Recall the notation $\ro(A)$ and $\co(A)$ from (\ref{ro-co}), we set
\begin{align} \label{Xi-set}
\begin{split}
\Xi^0 &=\{ A\in \Xi|\
 \ro (A)_{n+1} > 0, \ \co (A)_{n+1}> 0  \}\times \{0\},
 \\
\Xi^+ &=\Xi\backslash \Xi^0 \times \{+\},
\\
\Xi^- &=\Xi\backslash \Xi^0\times \{ - \}.
\end{split}
\end{align}
For convenience, we sometimes write $A^{\pm}$ for $(A,\pm )\in\Xi^{\pm}$ and $A^0 $ for $(A, 0)$ in $\Xi^0$.
We further set
\begin{align}\label{Xi}
\Xi_{\mbf D}=\Xi^+\sqcup \Xi^0 \sqcup \Xi^-.
\end{align}
Elements in $\Xi_{\mbf D}$ will be called $signed$ $matrices$.  We have

\begin{lem}
The map $\tilde \Phi$ in (\ref{tphi}) induces a bijection
\begin{align}
\label{Phi}
\Phi: G\backslash \mscr X\times \mscr X\to \Xi_{\mbf D}, \quad
G.(V, V') \mapsto \A \equiv (A,\alpha),
\end{align}
such that  $\tilde \Phi(G.(V,V') )=A$ and
$\alpha=\sgn (i, j)$ if $(V, V') \in \mscr X^i\times \mscr X^j$.
Moreover, we have
\begin{equation} \label{Xi-number}
  \begin{split}
\# \Xi_{\mbf D}=\binom{2n^2+2n+d}{d}+2\binom{2n^2+n+d-1}{d}-\binom{2n^2+d-1}{d}.
\end{split}
\end{equation}
\end{lem}

\begin{proof}
  The first part follows from  (\ref{parity}) and the definition of $\Xi_{\bf D}$. We now calculate $\#\Xi_{\bf D}$.
  From (\ref{Xi-set})
  $\#\Xi_{\bf D}=\#\Xi+\#\Xi^-$. We have
  \begin{equation}
  \label{Xi-1}
  \begin{split}
    \#\Xi&=\#\left
    \{a_{ij}, i\in [1,n], \forall j; a_{n+1,j},j\in [1,n]|\sum_{ i\leq n; j}a_{ij} +\sum_{j\leq n}a_{n+1,j}=d-\frac{a_{n+1,n+1}}{2} \right
    \}\\
    &=\sum_{l=0}^d\binom{2n^2+2n+d-l-1}{d-l}=\binom{2n^2+2n+d}{d}.
  \end{split}
  \end{equation}
  Denote $\Xi^-_1=\{A\in \Xi^-|a_{n+1,j}=0,\forall j\}$ and $\Xi^-_2=\{A\in \Xi^-|a_{i,n+1}=0,\forall i\}$.
  Then $\Xi^-=\Xi^-_1\sqcup \Xi^-_2$, and
  \begin{equation}
  \label{Xi-2}
    \#\Xi^-=\#\Xi^-_1+\#\Xi^-_2-\#\Xi^-_1\cap \Xi^-_2=2\binom{2n^2+n+d-1}{d}-\binom{2n^2+d-1}{d}.
  \end{equation}
  Lemma follows from (\ref{Xi-1}) and (\ref{Xi-2}).
\end{proof}

\subsection{Multiplication formulas in $\mcal S$}\label{sec4.3}

For each signed matrix $\A \in \Xi_{\mbf D}$, we denote by $\mathcal O_{\A}$ the associated $G$-orbit.
We introduce the following notations.
\begin{align} \label{Anot}
\begin{split}
\sup(\A) & = (i, j),\quad {\rm if}\ \mcal O_{\A} \subseteq  \mscr X^i\times \mscr X^j. \\
\sgn(\A) & = \sgn (\sup(\A)). \\
\ro (\A)   & = \ro (A), \\
 \co (\A) & = \co (A), \\
  \ur (\A) &= \ur (A), \\
  \A+B    &= A+B \quad {\rm if} \  \A=(A,\alpha) \ \mbox{and}\ B \ \mbox{a matrix}. \\
 \p( \A)   & =
\begin{cases}
1 & {\rm if} \ \ur (\A) \ {\rm is \ odd},\\
0 & {\rm otherwise}.
\end{cases}
\end{split}
\end{align}
We note that $\A +B$ is a matrix instead of a signed matrix.
For a signed matrix $\A \in \Xi_{\mbf D}$, we define
\begin{equation}\label{eq68}
\begin{split}
&s_l(\A)=
\left\{\begin{array}{ll}
    1 & {\rm if}\ {\rm ro}(\A)_{n+1}>0,\\
    2 & {\rm if} \ {\rm ro} (\A)_{n+1} =0, \sgn(\A) =+,\\
    3 & {\rm if} \ {\rm ro} (\A)_{n+1} =0, \sgn (\A)  = -.
  \end{array}
  \right.
  \\
& s_r(\A)=\left\{\begin{array}{ll}
    1 & {\rm if}\ {\rm co}(\A)_{n+1}>0,\\
    3- \p(\A)\delta_{0,{\rm ro} (\A)_{n+1}} & {\rm if} \ {\rm co} (\A)_{n+1} =0, \sgn(\A)  = -,\\
    2+ \p(\A)\delta_{0,{\rm ro} (\A)_{n+1}} &{\rm if} \ {\rm co} (\A)_{n+1} =0, \sgn (\A) = +.
  \end{array}
  \right.
\end{split}
\end{equation}
Then, we have
\begin{lem}
$\sup(\A)=( s_l(\A), s_r(\A))$ and $\sgn (\A) = \sgn (s_l(\A), s_r(\A))$,  for all $\A \in \Xi_{\mbf D}$.
\end{lem}

For any $n\in \mbb Z,  k\in \mbb N$,
 we set
$$(n)_v=\frac{v^{2n}-1}{v^2-1},
\quad {\rm and}\quad
\begin{pmatrix}
  n\\k
\end{pmatrix}_{\!\!\!{}v}= \prod_{i=1}^k\frac{(n+1-i)_v}{(i)_v}.
$$
Let
$$E_{ij}^{\theta}=E_{ij}+E_{N+1-i, N+1-j},$$
where $E_{ij}$ is the $N\times N$ matrix whose $(i,j)$-entry is 1 and all other entries are 0.
Let $e_{\A }$  be the characteristic function of the $G$-orbit corresponding to $\A\in \Xi_{\mbf D}$.
It is clear that  the set $\{e_{\A  }| \A \in \Xi_{\mbf D}\}$ forms a basis of $\mcal S$.
For convenience, we set
\[
e_{\A} =0,\quad \mbox{if}\  \A\not \in  \Xi_{\mbf D}.
\]
Recall the notations, such as  $\A + B$,  from (\ref{Anot}). We have

\begin{prop}\label{prop4}
Suppose that $\Aa\equiv A^{\alpha}$,  $\B$, $\C   \in \Xi_{\mbf D}$ and $h\in [1, n]$.

$(a)$ If $\B$ is chosen such that  $\B-E_{h,h+1}^{\theta}$  is  diagonal,
 ${\rm co}(\B)={\rm ro}(\A)$ and
$s_r(\B)=s_l(\Aa)$, then
\begin{align} \label{34}
e_{\B} * e_{\A}=\sum_{p\in[1,N]}  v^{2\sum_{j>p}a_{hj}}  & (1+ a_{hp})_v\
 e_{\A_p}, \quad \mbox{where}\\
 & \A_p =(A+E^{\theta}_{hp}-E^{\theta}_{h+1,p}, \sgn (s_l(\B), s_r(\Aa))) \in \Xi_{\mbf D}. \nonumber
 \end{align}

$(b)$ If $h\neq n$ and $\C$ has that  $\C-E_{h+1,h}^{\theta}$ is diagonal, ${\rm co}(\C)={\rm ro}(\A)$ and
$s_r(\C)=s_l(\Aa)$,
then
\begin{align} \label{35}
e_{\C} *  e_{\A}=\sum_{1\leq p\leq N} v^{2\sum_{j<p}  a_{h+1,j}}  & (1+a_{h+1,p})_v\
e_{\A(h, p)}, \quad \mbox{where}\\
& \A(h, p) = (A-E^{\theta}_{hp}+E^{\theta}_{h+1,p}, \sgn (s_l(\C), s_r(\Aa))) \in \Xi_{\mbf D}. \nonumber
\end{align}

$(c)$
If the condition $h\neq n$ in (b) is replaced by $h=n$, then we have
\begin{align}\label{eq33}
e_{\C}   * e_{\A}  = &
\sum_{p\neq n+1, a_{n,p}\geq 1}v^{2\sum_{j<p}a_{n+1,j}}  (1+a_{n+1,p})_v\
e_{\A(n, p)} +\\
& v^{2\sum_{j<n+1}a_{n+1,j}} ((1+a_{n+1,n+1})_v
+(1-\delta_{0, \ro(\A)_{n+1}})
v^{a_{n+1,n+1}})\
e_{\A(n, n+1)}. \nonumber
\end{align}
\end{prop}

\begin{rem}\label{remEx}
Although the formulas (\ref{34}), (\ref{35}) and (\ref{eq33}) look similar to formulas (3.9), (3.10) and (3.11) in ~\cite{BKLW13}, they are different  in many ways. For example, if we take $h=n$ and   $\A$ in (\ref{34}) to be such that $A - E^{\theta}_{n+1, n}$ is diagonal, then we have
\[
e_{\B} * e_{\A} =
\begin{cases}
(1+ a_{nn})_v \ e_{\A_n}, & \mbox{if}\ \sgn (\B) =\sgn (\A),\\
  e_{\A_{n+2}}, & \mbox{otherwise}.
\end{cases}
\]
While in ~\cite{BKLW13}, the product $e_B * e_A$ is a sum of the terms in  the right-hand side of the above identity in a similar situation.
\end{rem}

\begin{proof}
The proof of (a) and (b) is the same as the one for Lemma 3.2 in \cite{BLM90}.
We must  show (c),  which  can be  reduced to analogous  results at the  specialization of  $v$ to $\sqrt q$.
We first deal with the case when  $\ro(\A)_{n+1}>0$.
Let $V=(V_i)_{ 1\leq i\leq N}$ and $V'=(V'_i)_{ 1\leq i\leq N}$ be two flags  such that $(V,V')\in \mcal O_{\A'}$
for  a signed matrix  $\A'=(a'_{ij},  \sgn(s_l(\C), s_r(\A))$.
Set $Z=\{U| V_n\overset{1}{\subset} U\subset V_{n+1},  U\ \text{is isotropic}\}$.
Let $V_U$ be the flag obtained by replacing $V_n$ (resp. $V_{n+1}$) in $V$ by $U$ (resp. $U^{\bot}$).
Then $(V, V_U)\in \mcal O_{\C}$ if and only if $U\in Z$.
Let
$$Z_p=\{U\in Z|V_n\cap V_j'=U\cap V_j'\ {\rm if}\ j<p; V_n\cap V_j'\neq U\cap V_j'\ {\rm if}\ j\geq p\},
$$
so that
$\{Z_p|1\leq p\leq N\}$ form a partition of $Z$.
Moreover, if $U\in Z_p$ and $(V_U, V')\in \mcal O_{\A}$, then
$$a_{n,p}=a'_{n,p}+1\quad {\rm and}\quad a_{n+1,p}=a'_{n+1,p}-1.$$
Hence $\A' = \A(n,p)$.
In particular, we have
$$e_{\C}  * e_{\A}=\sum_p \#Z_p\ e_{\A(n,p)}.$$
Observe that
$$Z_p \simeq \{U\ \text{is isotropic}|V_n\overset{1}{\subset}U \subset V_n+V_{n+1}\cap V_p',\
U\not \subset V_n+V_{n+1}\cap V_{p-1}'\}.$$
If $p\leq n$, $V_n+V_{n+1}\cap V_p'$ is isotropic, then
$$\#Z_p=(q-1)^{-1}(q^{\sum_{1\leq j\leq p}a'_{n+1,j}}-q^{\sum_{1\leq j\leq p-1}a'_{n+1,j}})
=q^{\sum_{1\leq j< p}a_{n+1,j}}\frac{q^{1+a_{n+1,p}}-1}{q-1}.$$
This matches with the coefficient of the first term on the right-hand side of (\ref{eq33}) for $p\leq n$.

We now compute the number $\# Z_p$ for $p\geq n+1$. We  set $W_i= \frac{V_n+V^{\bot}_{n}\cap V'_{i}}{V_n}$ and
consider the following flags
$$\begin{array}{ll}
0\subset W_{N-p} \subset  W_{N-p+1}
\subset W_{p-1}
\subset W_{p}
\subset W_N,& {\rm if}\  p>n+1, \vspace{5pt}\\
0\subset W_n \subset  W_{n+1}
\subset W_N,& {\rm if}\  p=n+1.
\end{array}$$
So $Z_p\simeq \{U\ \text{is isotropic}\ |\ U\subset W_p, U\not \subset W_{p-1}, |U|=1\}$ if $p\geq n+1$.
From this observation and applying Lemma \ref{countingisotrop}, we have that $\# Z_p$ matches with the coefficients of the terms in (\ref{eq33}).
Therefore, we have  (c)
when $\ro (\A)_{n+1} >0$.

Finally, we assume that $\ro(\A)_{n+1} =0$. In this case,  $\sgn(\A) = +$ or $-$,  and hence
$e_{\C} e_{\A} (V,V')=\sum_{V_U} e_{\C} (V,V_U) e_{\A} (V_U,V')$ and $V_U$ runs as follows.
$$
\left\{\begin{array}{ll}
V_U\in \mscr X^2& {\rm if}\  \sgn (\A) =+, \\
V_U\in \mscr X^3& {\rm if}\ \sgn(\A)  = -.
\end{array}\right.
$$
Moreover, if $(V_U,V')\in \mcal O_{\A}$, then $|U|=d$ and $|V_n|=d-1$.
Given such $V_n$, there exists a unique maximum isotropic vector subspace $U$ such that $V_n  \overset{1}{\subset} U$ and
$|U \cap M_d| \equiv d \bmod 2$ (or $|U \cap M_d| \not\equiv d\bmod 2$, but not both),
where $M_d$ is the fixed maximum isotropic subspace in Section \ref{setup}.
In this case, the coefficient of $e_{\A(n, n+1)}$ is equal to 1 in both cases. Therefore, we have
(c) for the case $\ro (\A)_{n+1} =0$.
\end{proof}

Recall the notations from (\ref{Anot}). By Proposition \ref{prop4} and an induction process, we have the following corollary.

\begin{cor}\label{cor4}

Suppose that $\Aa=A^{\alpha}$, $\B$, $\C  \in \Xi_{\mbf D}$, $h\in [1, n]$ and $r\in \mbb N$.

$(a)$ If ${\rm co}(\B)={\rm ro}(\A)$, $s_r(\B)=s_l(\Aa)$
and  $\B-rE_{h,h+1}^{\theta}$  is diagonal,  then we have
\begin{align} \label{at}
e_{\B} * e_{\A}
=\sum_{t =(t_u)\in \mbb N^N: \sum_{u=1}^Nt_u=r}
v^{2\sum_{j>u}a_{hj}t_u}\prod_{u=1}^N\begin{pmatrix}a_{hu}+t_u \\
t_u\end{pmatrix}_{\!\!\!{}v} e_{\A_t}, \quad \mbox{where} \\
\A_t=\left (A+\sum_{u=1}^Nt_u(E^{\theta}_{hu}-E^{\theta}_{h+1,u}),\sgn (s_l(\B), s_r(\Aa)) \right ) \in \Xi_{\mbf D}. \nonumber
\end{align}

$(b)$  If $h\neq n$, ${\rm co}(\C)={\rm ro}(\A)$,  $s_r(\C)=s_l(\Aa)$ and  $\C-rE_{h+1,h}^{\theta}$ is diagonal,  then
\begin{align} \label{aht}
e_{\C} *  e_{\A}
=\sum_{t =(t_u)\in \mbb N^N: \sum_{u=1}^Nt_u=r}
v^{2\sum_{j < u}a_{h+1,j}t_u}\prod_{u=1}^N\begin{pmatrix}
a_{h+1,u}+t_u\\
t_u\end{pmatrix}_{\!\!\!{}v}e_{\A(h, t)}, \quad\mbox{where}\\
\A(h, t) =\left  (A-\sum_{u=1}^Nt_u(E^{\theta}_{hu}-E^{\theta}_{h+1,u}), \sgn (s_l(\C), s_r(\Aa)) \right ) \in \Xi_{\mbf D}. \nonumber
\end{align}

$(c)$ If the condition $h\neq n$ in (b) is replaced by $h=n$, then we have
\begin{align} \label{ann}
e_{\C} * e_{\A}
&= \sum_{t=(t_u)\in \mbb N^N: \sum_{u=1}^N t_u=r}  v^{\tilde \beta(t)} \ \mathcal G  \ e_{\A(n, t)},\quad  \mbox{where} \ \\
\tilde \beta(t)
& = 2\sum_{j<u}a_{n+1,j}t_u+2\sum_{N+1-j<u<j}t_jt_u+\sum_{u>n+1}t_u(t_u-1),  \nonumber \\
\mathcal G
&=
\prod_{u<n+1}
\begin{pmatrix}
a_{n+1,u}+t_u+t_{N+1-u}\\
t_u
\end{pmatrix}_{\!\!\!{}v}
\cdot
 \prod_{u> n+1}
\begin{pmatrix}
a_{n+1,u}+t_u \\
t_u\end{pmatrix}_{\!\!\!{}v}
\cdot
\prod_{i=0}^{t_{n+1}-1}  L_i,  \nonumber\\
L_i
& = \frac{(a_{n+1,n+1}+1+2i)_v+(1-\delta_{0,i}\delta_{0, \ro(\A)_{n+1} }) v^{a_{n+1,n+1}+2i}} {(i+1)_v}. \nonumber
\end{align}
\end{cor}

Note that $L_i\in \mcal A$ since $a_{n+1, n+1}$ is even.

\begin{proof}
The proof of $(a)$ and $(b)$ is the same as the one for Lemma 3.4 in
\cite{BLM90}.

We now show $(c)$ by induction on $r$.
We rewrite $\C$ as  $\C_r$  to emphasize the dependence on $r$.
Let $d_{t'} (\A)$ be the coefficient of $\A(n, t')$ in the product  $e_{\C_{r}}e_{\A}$.
Let $\underline{p} \in \mbb N^n$ be the vector whose $p$-th entry is 1 and  0 elsewhere.
The statement ($c$) is reduced to show that
for any $t=(t_1,t_2,\cdots, t_N)\in \mbb N^n$ such that  $\sum_ut_u=r+1$, we have
\begin{align}
\label{Ind}
\sum_{t', \underline p}  d_{t'} (\A) \  d_{\underline{p}}(\A(n, t')) = (r+1)_v \ d_t(\A),
\end{align}
where the sum runs over pairs  $(t', \underline p)$ such that $\sum t'_u= r$ and $t'+\underline p= t$.

We shall prove (\ref{Ind}) by induction. When $r=0$, the statement (\ref{Ind}) holds automatically.
We first deal with the case when  $\ro(\A)_{n+1}>0$.
By the induction assumption, we have
\begin{eqnarray*}
&& \sum_{t', \underline{p}} d_{t'}(\A) d_{\underline{p}}(\A(n, t'))
= \sum_{t', \underline{p}} v^{2 \sum_{j<u}a_{n+1,j}t'_u + 2\sum_{N+1-j<u<j}t_j't_u' +
\sum_{u>n+1}t_u'(t_u'-1)}\\
&&\cdot\prod_{u<n+1}
\begin{pmatrix}
a_{n+1,u}+t'_u+t'_{N+1-u}\\
t'_u
\end{pmatrix}_{\!\!\!{}v}
\prod_{i=0}^{t'_{n+1}-1}
\frac{(a_{n+1,n+1}+1+2i)_v+(1-\delta_{0,i}\delta_{0,\ro(\A)_{n+1}}) v^{a_{n+1,n+1}+2i}} {(i+1)_v}\\
&&\cdot \prod_{u> n+1}
\begin{pmatrix}
a_{n+1,u}+t'_u \\
t'_u\end{pmatrix}_{\!\!\!{}v} (a_{n+1,p}+1+t_p'+t_{N+1-p}')_v \ v^{2\sum_{j<p}a_{n+1,j}+t_j'+t_{N+1-j}'}.
\end{eqnarray*}
Since $t'+\underline{p}=t$, we have
\begin{equation*}\label{eq39}
   t_i=t_i'+\delta_{ip}.
\end{equation*}
We can  compute the quotient  $ \sum_{t', \underline{p}} d_{t'}(\A) d_{\underline{p}} (\A(n, t')) / d_t(\A)$.
We  first calculate the power of $v^2$ for each $p$ in this quotient, which is
\begin{eqnarray*}
  &&\sum_{j<u}a_{n+1,j}t'_u+\sum_{N+1-j<u<j}t_j't_u'+
  \sum_{u>n+1}t_u'(t_u'-1)/2+\sum_{j<p}(a_{n+1,j}+t_j'+t_{N+1-j}')\\
  &&-\sum_{j<u}a_{n+1,j}t_u-
  \sum_{N+1-j<u<j}t_jt_u-
  \sum_{u>n+1}t_u(t_u-1)/2\\
  &&=-\sum_{N+1-p<u<p}t_u-\sum_{N+1-j<p<j}t_j+\sum_{j<p}(t_j+t_{N+1-j})+\left\{
  \begin{array}{ll}
    -t_p& {\rm if}\ p>n+1\\
    0&{\rm otherwise}
  \end{array}
  \right.\\
  &&=\sum_{j<p}t_j.
\end{eqnarray*}
We then calculate the coefficients  containing $v$-numbers for each $p$ in the above quotient, which can be broken into the following three cases.
If $p<n+1$, then the coefficient involving $v$-numbers is
$$
(a_{n+1,p}+t_p+t_{N+1-p})_v\begin{pmatrix}
a_{n+1,u}+t_p+t_{N+1-p}-1\\
t_p-1
\end{pmatrix}_{\!\!\!{}v}/\begin{pmatrix}
  a_{n+1,u}+t_p+t_{N+1-p}\\
  t_p
\end{pmatrix}_{\!\!\!{}v}=(t_p)_v.
$$
If $p=n+1$, then the term is
$$
\frac{(t_{n+1})_v}{(a_{n+1,n+1}+2t_{n+1}-1)_v+q^{1/2a_{n+1,n+1}+t_{n+1}}}
((a_{n+1,n+1}+2t_{n+1}-1)_v+ v^{a_{n+1,n+1}+ 2 t_{n+1}})=(t_p)_v.$$
If $p>n+1$, then  the term is
$$
\frac{\begin{pmatrix}
  a_{n+1,N+1-p}+t_p+t_{N+1-p}-1\\
  t_{N+1-p}
\end{pmatrix}_{\!\!\!{}v}}{\begin{pmatrix}
  a_{n+1,N+1-p}+t_p+t_{N+1-p}\\
  t_{N+1-p}
\end{pmatrix}_{\!\!\!{}v}}
\frac{\begin{pmatrix}
a_{n+1,p}+t_p-1 \\
t_p-1\end{pmatrix}_{\!\!\!{}v}}{\begin{pmatrix}
a_{n+1,p}+t_p \\
t_p\end{pmatrix}_{\!\!\!{}v}} (a_{n+1,p}+t_p+t_{N+1-p})_v=(t_p)_v.
$$
Summing up, we  have
\begin{eqnarray*}
  \sum_{t', \underline{p}} d_{t'} (\A) d_{\underline{p}} (\A(n, t')) / d_t(\A)=\sum_p v^{2\sum_{j<p}t_j} (t_p)_v
  =\frac{\sum_p (v^{2\sum_{j\leq p}t_j}-v^{2\sum_{j\leq p-1}t_j})}{v^2-1}=(r+1)_v.
\end{eqnarray*}
This proves (\ref{Ind}) under the assumption that $\ro(\A)_{n+1} > 0$.
The proof of (\ref{Ind}) for the case of $\ro(\A)_{n+1}=0$ is similar and skipped.
\end{proof}

\subsection{$\mcal S$-action on $\mcal V$}

A degenerate version of Proposition \ref{prop4}  gives us  an explicit description of  the $\mcal S$-action on $\mcal V=\mcal A_G(\mscr X\times \mscr Y)$ as follows.
For any $r_j\in [1,N]$, we denote $\check{r}_j=r_j+1$ and $\hat{r}_j=r_j-1$.

\begin{cor}\label{cor9}
  For any $1\leq i\leq n$, we have
  \begin{align}
E_i *  e_{r_1\cdots r_D}
&=v^{-\sum_{1\leq j \leq D} \delta_{i+1,r_j}}
\sum_{p,\   r_p=i} v^{2\sum_{j<p}\delta_{i+1,r_j}} \
e_{r_1\cdots,\check{r}_p\cdots \hat{r}_{D+1-p} \cdots r_{D}}, \nonumber\\
F_i * e_{r_1\cdots r_D}
&=v^{-\sum_{1\leq j \leq D} \delta_{i,r_j}}
\sum_{p,\ r_p=i+1} v^{2\sum_{j>p}\delta_{i, r_j}}
e_{r_1\cdots,\hat{r}_p\cdots \check{r}_{D+1-p} \cdots r_{D}}, \nonumber \\
H_i^{\pm 1} *  e_{r_1\cdots r_D}
& = v^{\mp \sum_{1\leq j \leq D} \delta_{a, r_j}  }  e_{r_1\cdots r_D} \quad {\rm and}\nonumber\\
J_{+} * e_{r_1,\cdots, r_D} &=
\begin{cases}
e_{r_1,\cdots, r_D}, & r_i\neq n+1,\forall i, \#\{ j\in [1,d]| r_j\geq n+1\}  \ is \ even,\\
0, &\mbox{otherwise},
\end{cases}\nonumber \\
J_{-}  *  e_{r_1,\cdots, r_D} &=
\begin{cases}
e_{r_1,\cdots, r_D}, & r_i\neq n+1,\forall i, \#\{ j\in [1,d]| r_j\geq n+1\}  \ is \ odd,\\
0, &\mbox{otherwise},
\end{cases}
\nonumber\\
J_{0} * e_{r_1,\cdots, r_D} &=
\begin{cases}
e_{r_1,\cdots, r_D}, & r_i= n+1,\ \mbox{for some} \ i,\\
0, &\mbox{otherwise}.
\end{cases}
\nonumber
\end{align}
\end{cor}
\begin{proof}
Since the number of columns of the matrix associated  to $e_{r_1\cdots r_D}$ is $D=2d$, the second term in (\ref{eq33}) disappears when we calculate the $E_n$ action on $e_{r_1\cdots r_D}$.
The first two identities follow directly  from Proposition \ref{prop4}.
The last four  identities are straightforward.
\end{proof}

\subsection{Standard basis of $\mcal S$}\label{sec4.5}

In this subsection, we assume that the ground field is an algebraic closure $\overline{\mbb F}_q$ of $\mbb F_q$ when we talk about the dimension of a $G$-orbit or its stabilizer.
We set
\[
d(\A)={\rm dim} \ \mcal O_{\A}
\quad\mbox{and}\quad
r(\A)={\rm dim} \ \mcal O_{\B}, \quad \forall \A\in \Xi_{\mbf D},
\]
where $\B=(b_{ij})^{\epsilon}$ is the signed diagonal matrix  such that $b_{ii}=\sum_ka_{ik}$ and $\epsilon = \sgn (s_l(\A), s_l(\A))$.
Denote by ${\rm C}_{G}(V,V')$ the stabilizer of $(V,V')$ in $G$.

\begin{lem}
\label{dimension}
 We have
\begin{equation*}
\begin{split}
{\rm dim}\ {\rm C}_{G}(V,V')
& =\frac{1}{2} \left ( \sum_{i\geq k, j\geq l}a_{ij}a_{kl}-\sum_{i\geq n+1, j\geq n+1}a_{ij} \right ),
\quad  {\rm if}\ (V,V') \in \mcal O_{\A},\\
{\rm dim}\ \mcal O_{\A}
&=\frac{1}{2} \left (\sum_{i<k\ {\rm or}\ j<l}a_{ij}a_{kl}-\sum_{i<n+1 \ {\rm or}\ j<n+1}a_{ij} \right ),\\
d(\A)-r(\A)
&=\frac{1}{2} \left (\sum_{i\geq k, j<l}a_{ij}a_{kl}-\sum_{i\geq n+1>j}a_{ij} \right ).
\end{split}
\end{equation*}
\end{lem}

Notice that the above dimensions are independent of the sign of $\A$.

\begin{proof}
Let $Z_{ij}$ be subspaces of $\mbb F_q^D$ with $\dim Z_{ij}=a_{ij}$ such that $V_r=\oplus_{i\leq r,j}Z_{ij}$, $V'_s=\oplus_{i, j\leq s}Z_{ij}$ for all $r,s\in [1,n]$ and $\mbb F^D_q=\oplus_{i,j}Z_{ij}$.
With respect to the decomposition,  an endomorphism $T$ of $\mbb F_q^D$ is determined by a family of linear maps $T_{(ij), (kl)}: Z_{ij}\rightarrow Z_{kl}$.
Similar to ~\cite[3.4]{BKLW13}, the Lie algebra of ${\rm C}_{G}(V,V')$ is the space of such $T$ satisfying the following conditions.
\begin{itemize}
  \item[(a)] $T_{(ij),(kl)}\neq 0$ implies that $i\geq k$ and $j\geq l$;
  \item[(b)] $T_{(ij),(kl)}=-{}^tT_{(N+1-k,N+1-l),(N+1-i,N+1-j)}, \quad \forall i,j,k,l\in [1,N].$
\end{itemize}
Note that $T_{(ij),(kl)}=-{}^tT_{(ij),(kl)}$ if and only if $i+k=N+1$ and $j+l=N+1$.
In this case, $a_{ij}=a_{kl}$ and the dimension of such $T_{(ij),(kl)}$ is $\frac{1}{2}(a_{ij}a_{kl}-a_{ij})$, from which
the first equality follows.

By using $\dim G=\frac{1}{2}D(D-1)=\frac{1}{2}(\sum_{i,j,k,l}a_{ij}a_{kl}-\sum_{i,j}a_{ij})$, we have the second equality.
The third equality follows from the previous two equalities.
\end{proof}

For any $\A\in \Xi_{\mbf D}$, let
$$[\A]=v^{-(d(\A)-r(\A))}e_{\A}.$$
We define a bar involution `$-$' on $\mcal A$ by $\bar v= v^{-1}$.
By Lemma ~\ref{dimension},  Corollary \ref{cor4} can be rewritten in the following form.

\begin{cor}\label{cor5}
Suppose that $\Aa=A^{\alpha}$, $\B$, $\C  \in \Xi_{\mbf D}$, $h\in [1, n]$ and $r\in \mbb N$.

$(a)$ If ${\rm co}(\B)={\rm ro}(\A)$, $s_r(\B)=s_l(\Aa)$
and  $\B-rE_{h,h+1}^{\theta}$  is diagonal,  then we have
\begin{align}\label{eq22}
[\B] *  [\A]
& =\sum_{t: \sum_{u=1}^N t_u=r} v^{\beta(t)}\prod_{u=1}^N
\overline{\begin{pmatrix}a_{hu}+t_u\\ t_u \end{pmatrix}}_{\!\!\!{}v}\
[\A_{t}], \ \mbox{where} \\
\beta(t)
& = \sum_{j\geq l} a_{hj} t_l - \sum_{j>l} a_{h+1, j} t_l + \sum_{j<l} t_j t_l +
\frac{1}{2}\delta_{hn}(\sum_{j+l<N+1}t_jt_l+\sum_{j<n+1}t_j),  \nonumber\\
\A_t
& =\left (A+\sum_{u=1}^Nt_u(E^{\theta}_{hu}-E^{\theta}_{h+1,u}),\sgn (s_l(\B), s_r(\Aa)) \right ) \in \Xi_{\mbf D}. \nonumber
\end{align}

$(b)$  If $h\neq n$, ${\rm co}(\C)={\rm ro}(\A)$,  $s_r(\C)=s_l(\Aa)$ and  $\C-rE_{h+1,h}^{\theta}$ is diagonal,  then
\begin{align}\label{eq24}
[\C] * [\A]
& =\sum_{t: \sum_{u=1}^N t_u=r}v^{\beta'(t)}\prod_{u=1}^N
\overline{\begin{pmatrix} a_{h+1,u}+t_u\\ t_u\end{pmatrix}}_{\!\!\!{}v}
\ [\A(h, t) ], \ \mbox{where}\\
\beta'(t)
& =\sum_{j\leq l} a_{h+1,j} t_l - \sum_{j<l} a_{hj} t_l + \sum_{j<l} t_j t_l,  \nonumber\\
\A(h, t)
&=\left  (A-\sum_{u=1}^Nt_u(E^{\theta}_{hu}-E^{\theta}_{h+1,u}), \sgn (s_l(\C), s_r(\Aa)) \right ) \in \Xi_{\mbf D}. \nonumber
\end{align}

$(c)$  If the condition $h\neq n$ in (b) is replaced by $h=n$, then we have
\begin{align}\label{eq27}
[\C] * [\A]
&= \sum_{t: \sum_{u=1}^N t_u=r}v^{\beta''(t)} \ \overline{\mathcal G}\
[\A(n,t)], \ \mbox{where}\\
\beta''(t)&=\sum_{j\leq l}a_{n+1,j}t_l-\sum_{j<l} a_{nj} t_l+\sum_{j<l, j+l\geq N+1}t_jt_l
  + r(r-1)/2-\sum_{j<n+1}t_j (t_j-1) / 2, \nonumber \\
\mathcal G
&=
\prod_{u<n+1}
\begin{pmatrix}
a_{n+1,u}+t_u+t_{N+1-u}\\
t_u
\end{pmatrix}_{\!\!\!{}v}
\cdot
 \prod_{u> n+1}
\begin{pmatrix}
a_{n+1,u}+t_u \\
t_u\end{pmatrix}_{\!\!\!{}v}
\cdot
\prod_{i=0}^{t_{n+1}-1}  L_i,  \nonumber\\
L_i
& = \frac{(a_{n+1,n+1}+1+2i)_v+(1-\delta_{0,i}\delta_{0, \ro(\A)_{n+1} }) v^{a_{n+1,n+1}+2i}} {(i+1)_v}. \nonumber
\end{align}
\end{cor}

The proof involves lengthy mechanical computations and is hence skipped.

\subsection{Generators of $\mcal S$}
\label{sec-partial order}

Define a partial order $``\leq"$ on $\Xi_{\mbf D}$  by $\Aa\leq \B$ if $\mcal O_{\Aa} \subset \overline{\mcal O}_{\B}$.
For any $\Aa=(a_{ij})^{\alpha}$ and $\B=(b_{ij})^{\epsilon}$ in $\Xi_{\mbf D}$, we say that $\Aa \preceq \B$ if and only if $\alpha=\epsilon$ and
the following two conditions hold.
\begin{align}\label{partial-order}
 \sum_{r\leq i, s\geq j} a_{rs}  & \leq \sum_{r\leq i, s\geq j} b_{rs}, \quad \forall i<j.\\
\sum_{\substack{r\leq n,\\s>N+1-j}} a_{rs} & \equiv \sum_{\substack{r\leq n,\\s>N+1-j}} b_{rs}  \ \bmod 2, \quad \mbox{if}\\
 \ro(\A)_{n+1} & = \co(\A)_{n+1} = 0, \quad  \ro(\B)_{n+1} = \co (\B)_{n+1} =0, \nonumber \\
\sum_{\substack{i\leq r\leq N+1-i,\\ j\leq s\leq N+1-j}} a_{rs} &  =
 \sum_{\substack{i\leq r\leq N+1-i,\\ j\leq s\leq N+1-j}} b_{rs} = 0 \quad  \text{and}
\sum_{\substack{r<i,\\s>N+1-j}} a_{rs} = \sum_{\substack{r<i,\\s>N+1-j}} b_{rs}, \quad  \forall i, j\in [1,n].\nonumber
\end{align}
The  relation $``\preceq" $ defines a second partial order on  $\Xi_{\mbf D}$.
 We say that $\Aa \prec \B$ if $\Aa \preceq \B$ and at least one of the inequalities in (\ref{partial-order}) is strict.
 By ~\cite[Theorem 8.2.8]{BB05} and ~\cite[Lemma 3.8]{BKLW13}, we have the following lemma.
\begin{lem}
  $\Aa \leq \B$ if and only if $\Aa \preceq \B$ for any $\Aa, \B\in \Xi_{\mbf D}$.
\end{lem}

We shall denote by ``$[ \m]$+ lower terms" an element in $\mcal S$ which is equal to $[\m]$
plus a linear combination of $[\m']$ with $\m' \prec \m$.
By Corollary (\ref{cor5}), we have

\begin{cor} \label{cor7}
 Fix  positive integers  $r$, $c$ and $h$ with $c$ positive and even and  $h\in [1,n]$.

$(a)$ Assume that $\Aa=(a_{ij})^{\alpha} \in \Xi_{\mbf D}$ satisfies one of the following two conditions:
$$\begin{array}{ll}
(1)\ a_{hj}=0, \forall j\geq k, \; a_{h+1,k}=r,   \; a_{h+1,j}=0, \forall j>k,\; {\rm if}\ h<n;\\
(2)\ a_{nj}=0, \forall j\geq k, \; a_{n+1,k}=r+(r+c)\delta_{n+1,k}, \;  a_{n+1,j}=0, \forall j>k,\; {\rm if}\ h=n, k\geq n+1.
\end{array}$$
If   $\B $  is subject to  $\B-rE_{h,h+1}^{\theta}$ is diagonal,
$s_r(\B)=s_l(\Aa)$ and $\co(\B)=\ro(\Aa)$, then
$$[\B] * [\Aa]=[\A_{t(k)}] +{\rm lower\ terms},\quad \mbox{where}
\quad   t(k)_u = r\delta_{u, k} .$$

$(b)$ Assume that $\Aa=(a_{ij})^{\alpha} \in \Xi_{\mbf D}$ satisfies one of the following conditions:
$$\begin{array}{llll}
(1)\ a_{hj}=0, \forall j< k, & a_{hk}=r,  & a_{h+1,j}=0, \forall j\leq k,& {\rm if}\ h<n,  \; \mbox{or}\\
(2)\ a_{nj}=0, \forall j< k, & a_{nk}=r,  & a_{n+1,j}=0, \forall j\leq k, &{\rm if}\ h=n, k\leq n.
\end{array}$$
If $\C$ satisfies that   $\C-rE_{h+1,h}^{\theta}$ is  diagonal,
 $s_r(\C)=s_l(\Aa)$ and $\co(\C)=\ro(\Aa)$, then
$$
[\C] * [\Aa]=[
\A(h, t(k)) ]+{\rm lower\ terms}.
$$
\end{cor}

We define an order on $\mbb N\times \mbb N$ by
\begin{equation}\label{order}
  (i,j)<(i',j')\quad \text{if and only if}\quad  j'-i'<j-i\ {\rm or}\  j'-i'=j-i, i'<i.
\end{equation}
By using Corollary ~\ref{cor7}, we are able to prove the following  theorem.

\begin{thm}\label{thm1}
For any $\Aa=(a_{ij})^{\alpha} \in \Xi_{\mbf D}$, we set $R_{ij}=\sum_{k=1}^ia_{kj}$.
There exist signed matrices $\m(i, j)$  such that
$\m(i, j) - R_{ij}E^{\theta}_{i,i+1}$ is diagonal   and
$$\prod_{1\leq i<j \leq N} [\m (i, j) ]  = [\Aa]+{\rm lower\ terms},$$
where the product is taken in the order  (\ref{order}). The product has $N(N-1)/2$ terms.
\end{thm}

\begin{proof}
We show the theorem  for $n=2$.
Let $B_{10}$ be a diagonal matrix with diagonal entries being $(\co(\A)_1,\cdots, \co(\A)_5)$.
We set
\[
\B_{10} = (B_{10}, \sgn (s_r(\A), s_r(\A))).
\]
 For $i=1,\cdots, 4$,
let $B_{1i}$ be the matrix  such that $B_{1i}-R_{i,i+1}E^{\theta}_{i,i+1}$ is a diagonal matrix and ${\rm co}(B_{1i})={\rm ro}(B_{1,i-1})$.
We set
\[
\B_{1 i} = (B_{1, i}, \sgn ( s_l(\B_{1, i}), s_r(\B_{1, i})) ),
\]
where $s_l(\B_{1, i})$ and $s_r(\B_{1, i})$ are defined inductively by
\begin{equation}\label{eqsupp1}
  \begin{split}
    s_r(\B_{1, i})&= s_l(\B_{1, i-1}), \ \forall i\in [1,4],\\
    s_l(\B_{1, i})&=\begin{cases}
      1, &{\rm if}\  \ro(\B_{1, i})_{n+1} \neq 0,\\
      s_r(\B_{1, i})+(-1)^{s_r(\B_{1, i})}\p(\B_{1, i}),&{\rm if}\   \ro(\B_{1, i})_{n+1} = 0 = \co(\B_{1, i})_{n+1},\\
      2\ {\rm or}\ 3, & {\rm if} \   \ro(\B_{1, i})_{n+1} = 0, \ \co(\B_{1, i})_{n+1}\neq 0.
    \end{cases}
  \end{split}
\end{equation}
We note that $s_l(\B_{1, i})$ has multiple choices in some cases.
When this case happens, we always set $s_l(\B_{1, i})=2$.
By Corollary \ref{cor7}, we have
\begin{align*}
\begin{split}
[\B_{14}] * [\B_{13}] *  & [\B_{12}] * [\B_{11}] * [\B_{10}]   =[\A_1]+{\rm lower\ terms},\quad
\mbox{where} \\
&\A_1=(A_1, \sgn(s_l(\B_{1, 4}), s_r(\B_{1, 0}))) \ \mbox{with} \
A_1 =\begin{pmatrix}
    *&R_{12}&0&0&0\\R_{45}&*&R_{23}&0&0\\
    0&R_{34}&*&R_{34}&0\\0&0&R_{23}&*&R_{45}\\0&0&0&R_{12}&*
  \end{pmatrix},
\end{split}
\end{align*}
and  the $*$s in the diagonal  are some nonnegative integers  uniquely  determined by ${\rm co}(A_1)={\rm co}(B_{10})$.
Now let $B_{ji}$ be the matrices such that $B_{ji}-R_{i,i+j}E^{\theta}_{i,i+1}$ is a diagonal matrix
and ${\rm co}(B_{j,i})={\rm ro}(B_{j,i-1})$ for all $i\in [1, 5-j], j\in [2,4]$. Here we assume that $B_{j0}=B_{j-1,6-j}$.
We set
\[
\B_{ji} = (B_{ji}, \sgn(s_l(\B_{ji}), s_r(\B_{ji}))),
\]
where $s_l(\B_{ji})$ and $ s_r(\B_{ji})$ are defined in a similar way as (\ref{eqsupp1}) and
$s_l(\B_{ji})=2$ if it has multiple choices.
By repeating the above process, we have
$$
[\B_{41}] * [\B_{32}] * [\B_{31}] * [\B_{23}] * [\B_{22}] * [\B_{21}] * [\B_{14}] * [\B_{13}] * [\B_{12}] * [\B_{11}] * [\B_{10}]
=[\A]+{\rm lower\ terms}.
$$
Theorem follows for $n=2$. The general case can be shown similarly.
\end{proof}

We have immediately

\begin{cor} \label{S-M}
The products  $\m_{\A}=\prod_{1\leq i<j \leq N} [\m (i, j) ] $ for any  $\A\in \Xi_{\mbf D}$ in Theorem ~\ref{thm1} form  a basis for $\mcal S$.
(It is called a monomial basis of $\mcal S$.)
\end{cor}

By (\ref{eq22}), (\ref{eq24}) and (\ref{eq27}) and Corollary ~\ref{S-M}, we have

\begin{cor} \label{S-1}
The algebra $\mcal S$ (resp. $\mbb Q(v) \otimes_{\mcal A} \mcal S$) is generated by the elements $[\mathfrak e]$ such that
$\mathfrak e - RE^{\theta}_{i, i+1}$ (resp. either $\mathfrak e$ or  $\mathfrak e-E^{\theta}_{i, i+1}$)  is diagonal
for some $R\in \mbb N$ and $i\in [1, N-1]$.
\end{cor}

Observe that
\[
E_i =\sum [\C], \ F_i = \sum [\B], \ H^{\pm 1}_a =\sum  v^{\mp d_a} [\D],\quad \forall i\in [1, n], a\in [1, n+1],
\]
where $\B$,  $\C$ and $\D$ run over all signed matrices in $\Xi_{\bf D}$ such that $\B-E^{\theta}_{i, i+1}$,  $\C-E^{\theta}_{i+1, i}$ and $\D$ are diagonal, respectively, and  $d_a$ is  the $(a, a)$-entry of the matrix in $\D$.
We have  the following corollary by Corollary ~\ref{S-1}.

\begin{cor} \label{S-generator}
The algebra $\mbb Q(v) \otimes_{\mcal A} \mcal S$ is generated by the functions $E_i$, $F_i$, $H_a^{\pm 1}$, $J_{\alpha}$
for any $i\in [1, n]$, $a\in [1, n+1]$ and $\alpha\in \{ +, 0, -\}$.
\end{cor}

\begin{rem}
The order  (\ref{order})  in Theorem ~\ref{thm1} is different from the ones  in ~\cite[Theorem 3.6.1]{BKLW13} and ~\cite[3.9]{BLM90}.
It can be shown that using the latter orders, one can construct  a different monomial basis for the algebra $\mcal S$.
\end{rem}

\subsection{Canonical basis of $\mcal S$}

In this subsection, we  assume that the ground field is an algebraic closure  $\overline{\mbb F}_q$ of the finite field $\mbb F_q$.
Let $IC_{\A}$ be the intersection cohomology complex of $\overline{\mcal O}_{\A}$,
normalized  so that the restriction of $IC_{\A}$ to $\mcal O_{\A}$ is the constant sheaf on $\mcal O_{\A}$.
Since $IC_{\A}$ is a $G$-equivariant complex and the stabilizers of the points in $\overline{\mcal O}_{\A}$ are connected,
the restriction of the $i$-th cohomology sheaf $\mscr H_{\mcal O_{\B}}^i(IC_{\A})$
of $IC_{\A}$ to $\mcal O_{\B}$ for $\B\leq \A$ is a trivial local system.
We denote $n_{\B,\A,i}$ the rank of  this local system.
We set
\begin{equation}\label{eq35}
  \{\A\}=\sum_{\B \leq \A} P_{\B,\A}[\B ],\quad \mbox{where}\quad
   P_{\B,\A}=\sum_{i\in \mbb Z} n_{\B  , \A,i}v^{i-d(\A)+d(\B )}.
\end{equation}
The polynomials $P_{\B, \A}$ satisfy
\begin{equation}\label{eq36}
P_{\A,\A}=1\quad {\rm and}\quad P_{\B,\A}\in v^{-1}\mbb Z[v^{-1}]\ {\rm for \ any }\ \B <\A.
\end{equation}
Since $\{[\A]| \A\in \Xi_{\bf D}\}$ is an $\mcal A$-basis of $\mcal S$, by (\ref{eq35}) and (\ref{eq36}), we have

\begin{lem}
The set $\{\{\A\}| \A \in \Xi_{\mbf D}\}$ forms an $\mcal A$-basis of $\mcal S$,  called the $canonical$ $basis$.
\end{lem}

By the sheaf-function principle, we have

\begin{cor} \label{S-positivity}
The structure constants of $\mcal S$ with respect to the canonical basis $\{\{\A\}| \A\in \Xi_{\mbf D}\}$
are in $\mbb N[v, v^{-1}]$.
\end{cor}

\subsection{Inner product on $\mcal S$}

We shall define an inner product on $\mcal S$ following  ~\cite[Section 3]{M10} and \cite[3.7]{BKLW13}.
Since the arguments and statements are very similar, we shall be sketchy.

Denote by $^t\!A$ the transposition matrix of $A$.
For any signed matrix $\A =(A,\epsilon)$, we define $^t\! \A = (\ \!{}^t\!A, \epsilon')$
where
$$
\epsilon'=
\begin{cases}
\epsilon, & {\rm if} \sup (\A) \neq (2, 3),  (3,2),\\
-\epsilon, & {\rm otherwise}.
\end{cases}
$$


For any $\A\in \Xi_{\mbf D}$, we set
\begin{equation}
  d_{\A}=d(\A)-r(\A)\quad {\rm and}\quad \mscr X^V_{\A}=\{V'\in \mscr X|(V,V')\in \mcal O_{\A}\}.
\end{equation}
We define a bilinear form
$$(- , - )_D: \mcal S\times \mcal S\rightarrow \mcal A$$
by
$$(f_1, f_2)_D=\sum_{V,V' \in \mscr X}v^{\sum_i|V_i/V_{i-1}|^2-\sum_i|V'_i/V'_{i-1}|^2}f_1(V,V') f_2(V,V'),\quad \forall f_1, f_2\in \mcal S.$$
In particular,
$$(e_{\A},e_{\B})_D=\delta_{\A,\B}v^{2(d_{\A}-d_{^t\A})} \# \mscr X^V_{{}^t\!\A},\quad \forall \A,\B\in \Xi_{\mbf D},$$
where $V$ is any element in $\mscr X$ such that $|V_i/V_{i-1}|={\rm co}(\A)_i$.
By the definition of $d_{\A}$ and Lemma \ref{dimension}, we have
$$d_{\A}-d_{\ \! ^t\!\A}=\frac{1}{4} \sum_i({\rm ro}(\A)_i^2-{\rm co}(\A)_i^2)-\frac{1}{4}({\rm ro}(\A)_{n+1}-{\rm co}(\A)_{n+1}).$$
This implies that
\begin{equation}\label{eq43}
  d_{\A}-d_{\ \! ^t\!\A}+d_{\B}-d_{\ \! ^t\!\B}=d_{\C}-d_{\ \! ^t\!\C}
\end{equation}
if ${\rm ro}(\A)={\rm ro}(\C)$, ${\rm co}(\A)={\rm ro}(\B)$ and ${\rm co}(\B)={\rm co}(\C)$.
By using (\ref{eq43}) and the same argument as the one proving Proposition 3.2 in \cite{M10}, we have the following proposition.

\begin{prop} \label{biadjoint}
For any $\ma, \mb, \mc \in \Xi_{\mbf D}$, we have
$$([\ma]e_{\mb}, e_{\mc})_D=v^{d_{\ma} -d_{^t\ma}}(e_{\mb}, [\ \! ^t\!\ma]e_{\mc} )_D.$$
\end{prop}

Moreover, the following proposition  holds from Proposition ~\ref{biadjoint}.

\begin{prop}\label{prop11}
For any $\mb , \mc\in \Xi_{\mbf D}$, we have
\[
([\mb],[\mc])_D \in \delta_{\mb, \mc}+v^{-1}\mbb Z[v^{-1}],
\quad
\mbox{and}
\quad
(x \mb , \mc )_D =(\mb, \rho(x) \mc)_D, \quad \forall x\in \mcal S,
\]
where $\rho$ is defined in (\ref{rho}).
\end{prop}

We define a bar involution $\bar{}: \mcal S\to \mcal S$ by
\[
\bar{v} =v^{-1} \quad \mbox{and} \quad
\overline{[\mathfrak e]}=[\mathfrak e],
\]
for any $\mathfrak e$ in $\Xi_{\mbf D}$ such that $\mathfrak e-RE^{\theta}_{i, i+1}$ for some $R\in \mbb N$ and $i \in [1, N-1]$.
By an argument similar to ~\cite{BLM90} and ~\cite{BKLW13}, we have
$\overline{\{\A\}} =\{\A\}$. By Proposition ~\ref{prop11} and this observation, we have

\begin{cor}
The canonical basis $\{ \{\A\}|  \A \in \Xi_{\mbf D} \}$  of $\mbb Q(v)\otimes_{\mcal A} \mcal S $ is characterized up to sign by the properties:
\[
\{\A\} \in \mcal S, \quad
\overline{\{\A\}} =\{\A\}
\quad \mbox{and} \quad
( \{ \ma \} , \{ \ma' \})_D \in \delta_{\A, \A'} +v^{-1}\mbb Z[v^{-1}].
\]
\end{cor}

\section{The limit algebra $\mcal K$ and its canonical basis}
\label{algebra-KD}

We shall apply the stabilization process to the algebras $\mcal S$ in (\ref{S(X)})  as $D$ goes to $\infty$, following ~\cite{BLM90}.
We write $\mcal S_D$ to emphasize the dependence on $D$, and
$\Xi_{\mbf D}(D)$ for the set $\Xi_{\mbf D}$ in (\ref{Xi}) for the same reason.

\subsection{Stabilization} \label{Stab}

Let $I'=I-E_{n+1,n+1}$, where $I$ is the identity matrix. We set
$$
{}_pA=A+2pI'
\quad \mbox{and} \quad
{}_p\A=( {}_pA, \alpha), \quad \mbox{if}\ \A=(A,\alpha).
$$
Let
\begin{align}\label{tXi}
\widetilde{\Xi}_{\mbf D} =\{\A=(A, \alpha ) \in {\rm Mat}_{N\times N}(\mbb Z)\times \{+, 0, -\}  | {}_p \A  \in \Xi_{\mbf D}(D) \ \text{for some}\ p\in \mbb Z, D\in \mbb N\}.
\end{align}
For any matrix  $\A\in \wt{\Xi}_{\mbf D}$,
the notations introduced  in (\ref{Anot}) and (\ref{eq68}) are still well-defined and will be used freely in the following.
Moreover, we observe that  $\sgn(\A)=\sgn({}_p\A)$.
Let
$$\mcal K= \mbox{span}_{\mcal A} \{ [\A] | \A\in \wt{\Xi}_{\mbf D} \},
$$
where the notation $[\A]$ is a formal  symbol
bearing no geometric meaning.
Let $v'$ be a second indeterminate, and
\begin{align} \label{R}
\mfk R=\mbb Q(v)[v', v'^{-1}].
\end{align}
We have

\begin{prop}
\label{prop5}
Suppose that $\A_1, \A_2,\cdots, \A_r \ (r\geq 2)$ are signed matrices in $\widetilde{\Xi}_{\mbf D}$
such that ${\rm co}(\A_i)={\rm ro}(\A_{i+1})$ and
 $s_r(\A_i)=s_l(\A_{i+1})$ for $1\leq i \leq r-1$.
There exist $\z_1, \cdots, \z_m\in \widetilde{\Xi}_{\mbf D}$, $G_j(v,v')\in \mfk R$ and $p_0\in \mbb N$ such that in $\mcal S_D$ for some $D$,  we have
$$[{}_p \A_1] * [{}_p\A_2] * \cdots *[{}_p \A_r]=\sum_{j=1}^mG_j(v,v^{-p})[{}_p \z_j],\quad
\forall p \in 2\mbb N, p\geq p_0.$$
\end{prop}

\begin{proof}
The proof is essentially the same as the one for Proposition 4.2 in \cite{BLM90}
by using Corollary \ref{cor5} and Theorem \ref{thm1}.
The main difference is that  when  $h=n$,  the twists  $\beta(t)$ and  $ \beta''(t)$ in (\ref{eq22}) and (\ref{eq27}), respectively, change when $\A$ is replaced by
${}_p \A$.
To  remedy this difference,  we adjust these two twists as follows.
 $$
 \gamma(t)=\beta(t)- a_{nn}  \sum_{l\leq n} t_l\quad
{\rm and}\quad
\gamma''(t)=\beta''(t)+ a_{nn} \sum_{n<l}t_l,\quad \mbox{if}\  h=n.
$$
Then the new twists $\gamma(t)$ and $\gamma''(t)$ remain the same when $\A$ is replaced by ${}_p\A$.
For example, when $r=2$ and $\A_1$ is chosen such that
 $\A_1-RE_{n,n+1}^{\theta}$ is a diagonal with $R\in \mbb N$,
the structure constant $G_{t}(v,v')$ is defined by
$$G_{t}(v,v')=v^{\gamma(t)}\prod_{\overset{1\leq u\leq N}{u\neq n}}\overline{\begin{pmatrix}a_{nu}+t_u\\
  t_u\end{pmatrix}}_{\!\!\!{}v}\prod_{1\leq i\leq t_n}\frac{v^{-2(a_{nn}+t_n-i+1)}v'^2-1}{v^{-2i}-1}
  v^{\sum_{l\leq n}a_{nn}t_l}v'^{-\sum_{l\leq n}t_l}.$$
Similaryly,
if  $r=2$ and $\A_1$ is chosen such that
$\A_1-RE_{n+1,n}^{\theta}$ is  diagonal with $R\in \mbb N$,
the structure constant $G_{t}(v,v')$ is defined by
\begin{equation*}
\begin{split}
G_{t}(v,v')=&v^{\gamma''(t)}v^{-\sum_{n<l}a_{nn}t_l}
\prod_{1\leq u\leq N,u\neq n+1}\overline{\begin{pmatrix}a_{n+1,u}+t_u\\
  t_u\end{pmatrix}}_{\!\!\!{}v}\\
  &\prod_{1\leq t\leq t_{n+1}}\frac{v^{-2(a_{n+1,n+1}+t_{n+1}-i+1)}v'^2-1}{v^{-2i}-1}
  \cdot v'^{\sum_{n<l}t_l}.
\end{split}
\end{equation*}
For the case when $\A_1$ is chosen such that
$\A_1- R E^{\theta}_{h, h+1}$ or $\A_1-RE^{\theta}_{h+1, h}$  is diagonal for some  $h<n$,
then  the structural constant $G_t(v,v')$ is defined  similarly as that in the proof of Proposition 4.2 in \cite{BLM90}, i.e.,
\[
G_{t}(v,v')=v^{\beta(t)}\prod_{\overset{1\leq u\leq N}{u\neq h}}\overline{\begin{pmatrix}a_{hu}+t_u\\
  t_u\end{pmatrix}}_{\!\!\!{}v}\prod_{1\leq i\leq t_h}\frac{v^{-2(a_{hh}+t_h-i+1)}v'^2-1}{v^{-2i}-1},
\]
for $\A_1$ such that $\A_1- R E^{\theta}_{h, h+1}$ is diagonal for some $h<n$,
and
\[
G_{t}(v,v')=v^{\beta'(t)}
\prod_{1\leq u\leq N,u\neq h+1}\overline{\begin{pmatrix}a_{h+1,u}+t_u\\
  t_u\end{pmatrix}}_{\!\!\!{}v}\\
  \prod_{1\leq t\leq t_{h+1}}\frac{v^{-2(a_{h+1,h+1}+t_{h+1}-i+1)}v'^2-1}{v^{-2i}-1},
\]
for $\A_1$ such that $\A_1- R E^{\theta}_{h+1, h}$ is diagonal for some $h<n$.
Bearing in mind  the above modifications, the rest of the proof for Proposition 4.2 in \cite{BLM90} can be repeated here.
\end{proof}

By specialization $v'$ at $v'=1$, we have

\begin{cor}
\label{cor6}
 There is a unique associative $\mcal A$-algebra structure on $\mcal K$, without unit,  where
 the product is given by
 $$[\A_1] \cdot [\A_2]\cdot \dots \cdot  [\A_r] =\sum_{j=1}^m G_j(v,1)[\z_j]$$
 if $\A_1,\cdots, \A_r$ are as in Proposition \ref{prop5}.
 \end{cor}


By corollary ~\ref{cor6} and comparing the $G_t(v,1)$'s with (\ref{eq22}), (\ref{eq24}) and (\ref{eq27}), the structure of $\mcal K$ can be determined by the following multiplication formulas.
Recall the notations from (\ref{Anot}).

Let $\A$ and  $\B \in \wt{\Xi}_{\mbf D} $ be chosen such that $\B -rE_{h,h+1}^{\theta}$
is  diagonal for some $1\leq h\leq n, r\in \mbb N$
satisfying ${\rm co}(\B)={\rm ro}(\A)$ and $s_r(\B)=s_l(\A)$. Then we have
\begin{equation}\label{eq58}
[ \B] \cdot  [\A]
 =\sum_{t} v^{\beta(t)}\prod_{u=1}^N
\overline{\begin{pmatrix}a_{hu}+t_u\\ t_u \end{pmatrix}}_{\!\!\!{}v}\
[\A_{t}],
\end{equation}
where the sum is taken over all $t=(t_u)\in \mbb N^N$ such that
$\sum_{u=1}^Nt_u=r$,
$\beta(t)$ is defined in (\ref{eq22}),   and
$\A_{t} \in \widetilde{\Xi}_{\mbf D}$ is in (\ref{at}).

Similarly, if $\A, \C \in \widetilde{\Xi}_{\mbf D}$  are chosen such that $\C-rE_{h+1,h}^{\theta}$
is diagonal  for some $1\leq h< n, r\in \mbb N$ satisfying ${\rm co}(\C)={\rm ro}(\A)$ and $s_r(\C)=s_l(\A)$, then we have
\begin{align}\label{eq57}
[\C] \cdot [\A]
 =\sum_{t}v^{\beta'(t)}\prod_{u=1}^N
\overline{\begin{pmatrix} a_{h+1,u}+t_u\\ t_u\end{pmatrix}}_{\!\!\!{}v}
\ [\A(h, t) ],
\end{align}
where  the sum is  taken over all $t=(t_u)\in \mbb N^N$ such that $\sum_{u=1}^Nt_u=r$,
$\beta'(t)$ is defined in (\ref{eq24}),
and
$\A(h, t) \in \widetilde{\Xi}_{\mbf D}$ is in (\ref{aht}).

If $\A, \C \in \widetilde{\Xi}_{\mbf D}$ are chosen  such that  $\C-rE_{n+1,n}^{\theta}$ is  diagonal  for some $r\in \mbb N$
satisfying ${\rm co}(\C)={\rm ro}(\A)$ and $s_r(\C)=s_l(\A)$, then we have
\begin{equation} \label{eq56}
[\C] \cdot [\A]
= \sum_{t: \sum_{u=1}^N t_u=r}v^{\beta''(t)} \  \overline{\mathcal G}\
[\A(n,t)],
\end{equation}
where
the sum is taken over all $t=(t_u)\in \mbb N^N$ such that $\sum_{u=1}^Nt_u=r$,
$\mathcal G$ and $\beta''(t)$ are in  (\ref{eq27}) and $\A(n, t)\in \wt{\Xi}_{\mbf D}$.

Given $\A, \A'\in \widetilde{\Xi}_{\mbf D}$, we shall denote $\A' \sqsubseteq \A $ if
$\A' \preceq \A$, ${\rm co}(\A')={\rm co}(\A)$, ${\rm ro}(\A')={\rm ro}(\A)$, $s_l(\A')=s_l(\A)$ and $s_r(\A')=s_r(\A)$.

By using (\ref{eq58}), (\ref{eq57}) and (\ref{eq56}) and arguing in a similar way as the proof of Theorem ~\ref{thm1}, we have

\begin{prop} \label{K-monomial}
For any $\A\in \wt{\Xi}_{\mbf D}$,
there exist signed matrices $\m(i,j)$  such that $\m(i,j) - R_{ij} E^{\theta}_{i, i+1}$ is diagonal with $R_{ij} =\sum_{k=1}^i a_{kj}$ and
\begin{equation} \label{eq59}
\mathfrak m_{\A} \equiv  \prod_{1\leq i<j \leq N}[\m(i, j) ]=[\A]+\sum_{\A'\sqsubseteq \A, \A'\neq \A}
\gamma_{\A', \A} [\A'],
\end{equation}
where $\gamma_{\A', \A}\in \mcal A$ and the product is taken in the order (\ref{order}).
\end{prop}

As  a consequence of the above proposition, we have

\begin{prop}
The algebra $\mcal K$ is generated by the elements $[\mathfrak e]$ such that
$\mathfrak e - rE^{\theta}_{i,i+1}$ is diagonal for some $r\in \mbb N$ and $i$ such that $1\leq i<  N$.
\end{prop}

We set
\begin{align}
[k]_v = \frac{v^{k} -v^{-k}}{v-v^{-1}}.
\end{align}
By applying  (\ref{eq58}), (\ref{eq57}) and (\ref{eq56}), we have
 \begin{align} \label{divided}
[\mathfrak e]  \cdot [\mathfrak e'] = [r+1]_v   [\mathfrak e'' ],
\end{align}
if
$\mathfrak e- E^{\theta}_{i, i+1}$, $\mathfrak e'-rE^{\theta}_{i, i+1}$, and $\mathfrak e'' -(r+1) E^{\theta}_{i,i+1}$ are diagonal for some $i\in [1, N-1]$,
$s_r(\mathfrak e)= s_l(\mathfrak e')$,
$s_l(\mathfrak e)= s_l(\mathfrak e'')$ and
$s_r(\mathfrak e') =s_r(\mathfrak e'')$.
From this observation, we have the following corollary.

 \begin{cor}
 The algebra $\mbb Q(v)\otimes_{\mcal A} \mcal K$ is generated by the elements
 $[\mathfrak e]$  such that either $\mathfrak e$ or  $\mathfrak e - E^{\theta}_{i,i+1}$ is diagonal
 for some   $i$ such that $1\leq i<N$.
 \end{cor}

\subsection{Bases of $\mcal K$}

We define a bar involution $^- \ : \mcal K\to \mcal K$ by
\[
\bar v = v^{-1}, \quad
\overline{[\mathfrak e]} = [\mathfrak e],
\]
for any $\mathfrak e$ such that  $\mathfrak e - RE^{\theta}_{i, i+1}$  is diagonal for some $R\in \mbb N$ and $i\in [1, N-1]$.
By using (\ref{eq59}), we have
\[
\overline{[\A]} =[\A] + \sum_{\A': \A' \sqsubseteq \A, \A'\neq \A} c_{\A', \A} [\A'], \quad \mbox{for some} \ c_{\A', \A} \in \mcal A.
\]

By a standard  argument similar to  the proof of Proposition 4.7 in \cite{BLM90}, we have the following proposition.

\begin{prop}\label{prop6}
For any $\A\in \widetilde{\Xi}_{\mbf D}$, there exists a unique element $\{\A\}$ in $\mcal K$ such that
$$\overline{\{\A\}}=\{\A\},\quad \{\A\}=[\A]+\sum_{\A'\sqsubset \A, \A'\neq \A}\pi_{\A', \A}[\A'],
\quad \pi_{\A', \A} \in v^{-1} \mbb Z [v^{-1}].$$
\end{prop}


By Propositions ~\ref{K-monomial} and ~\ref{prop6}, we have

\begin{cor}
The algebra $\mcal K$ possesses a standard basis $\{[\A] | \A \in \widetilde \Xi_{\mbf D} \}$, a monomial basis $\{\m_{\A} | \A\in \widetilde \Xi_{\mbf D}\}$ and a canonical basis $\{ \{ \A\} | \A\in \widetilde \Xi_{\mbf D}\}$.
\end{cor}

\subsection{From $\mathcal K$ to $\mcal S$}

We define an $\mcal A$-linear map
\begin{align} \label{Psi}
\Psi: \mcal K \to \mcal S
\end{align}
by
\[
\Psi([\A]) =
\begin{cases}
[\A], & \mbox{if} \ \A\in \Xi_{\mbf D},\\
0, & \mbox{otherwise}.
\end{cases}
\]
By comparing the multiplication formulas (\ref{eq22}), (\ref{eq24}) and (\ref{eq27})
with (\ref{eq58}), (\ref{eq57}) and (\ref{eq56}), respectively,   and following an argument in ~\cite{F12} and  ~\cite[Appendix A]{BKLW13}, we have

\begin{thm} \label{KS}
The map $\Psi$ in (\ref{Psi})  is a surjective algebra homomorphism. Moreover we have
\[
\Psi(\{\A\}) =
\begin{cases}
\{ \A\}, & \mbox{if} \ \A\in \Xi_{\mbf D},\\
0, & otherwise.
\end{cases}
\]
\end{thm}

Now the algebra $\mcal K$ acts on the $\mcal A$-module $\mcal V$ in (\ref{V}) via $\Psi$ and the $\mcal S$-action. By Lemma \ref{eq34}, we have

\begin{prop}
The algebra $\mcal K$ and $\mcal H_{\mscr Y}$ form a double centralizer, i.e.,
\[
\End_{\mcal K} (\mcal V) \simeq \mcal H_{\mscr Y},
\quad \mbox{if} \ n\geq d,\  \mbox{and} \quad \mcal K\to \End_{\mcal H_{\mscr Y}}(\mcal V)
\mbox{ is surjective.}
\]
\end{prop}

\subsection{Towards a   presentation of $\mbb Q(v)\otimes_{\mcal A} \mcal K$}\label{sec5.3}

We make an observation of the signed diagonal matrices in $\wt{\Xi}_{\mbf D}$ in (\ref{tXi}).
We denote by $D_{\lambda}$ the diagonal matrix whose $i$-th diagonal entry is $\lambda_i$,
for any $\lambda=(\lambda_i) \in \mbb Z^N$. We have

\begin{lem} \label{diagonal}
Suppose that $\D=(D_{\lambda}, \epsilon)$ is a signed diagonal matrix in $\wt{\Xi}_{\mbf D}$. Then we have
$\lambda_i = \lambda_{N+1-i}$ and $\lambda_{n+1}\in 2\mbb N$. Moreover,
$$
\lambda_{n+1} =
\begin{cases}
0  & {\rm if} \ \sgn (\D) = \pm, \\
\geq 2 & {\rm if} \ \sgn (\D) =0.
\end{cases}
$$
\end{lem}

For any signed diagonal matrix $\md$, we set
\begin{equation*}
\begin{split}
&E_h \md  = [(\md - E^{\theta}_{h, h} + E^{\theta}_{h+1,h},\sgn(\md) )],\quad \forall h\in [1,n], \\
&F_h \md  = [(\md -E^{\theta}_{h+1, h+1} +  E^{\theta}_{h,h+1}, \sgn(\md)) ],\quad \forall h\in [1,n-1].
\end{split}
\end{equation*}
For a signed diagonal matrix $\md =(D_{\lambda}, 0)$ of sign $0$, we set
\begin{equation*}
F_n \md =
\begin{cases}
[( \md - E^{\theta}_{n+1, n+1} +  E^{\theta}_{n,n+1}, 0) ] &  \mbox{if} \  \lambda_{n+1} \geq 4,\\
[(\D -  E^{\theta}_{n+1, n+1} +  E^{\theta}_{n,n+1}, +)]
+[(\D -  E^{\theta}_{n+1, n+1} +  E^{\theta}_{n,n+1}, -)] & \mbox{if} \ \lambda_{n+1} =2.
\end{cases}
\end{equation*}

For any element $y\in \U$ in Section \ref{U-first} and singed diagonal matrix $\md$, we shall define the notation $y\md$.
We may assume that $y$ is homogeneous.
We assume that $x\md $ is defined for all  homogenous $x\in \U$ of degree strictly less than $y$ ,
then  we define
\begin{equation}\label{eq63}
\textstyle E_j x \md =
\sum [\mathfrak e_j] \cdot x\md,
 \quad \forall j\in [1,n],
\end{equation}
where the sum runs over all signed matrices $\mathfrak e_j$  in $\widetilde \Xi_{\mbf D}$
such that $\mathfrak e_j - E^{\theta}_{j+1,j}$ is diagonal.
Although an infinite sum,  there is only finitely many nonzero terms, hence well-defined.
Similarly, we can define $F_j x \md $ for any $j\in [1,n]$.
Therefore, the notation $y \md$ for $y\in \U$ is well-defined.

\begin{prop}\label{prop13}
 For any signed diagonal matrices
 $\md = (D_{\lambda},\epsilon)$, $\md' =(D_{\lambda'}, \epsilon')$  in $\widetilde \Xi_{\mbf D}$, we have
the following relations in $\mcal K$.
\allowdisplaybreaks
\begin{eqnarray}
\label{i}
&& \md \md' =\delta_{\md, \md'} \md. \\
\label{ii}
&& \md'  E_n \md =0, \quad  {\rm if}\ \sgn (\md') =\pm, \\
&& \md'  F_n \md=0, \quad  {\rm if}\ \sgn(\md) = \pm, \nonumber \\
&&  \md E_h [(\md - E^{\theta}_{h, h} + E^{\theta}_{h+1,h+1}, \sgn (\md) ) ] = E_h [(\md - E^{\theta}_{h, h} + E^{\theta}_{h+1,h+1},\sgn(\md))], \nonumber \\
&& \md F_h [(\md + E^{\theta}_{h, h} - E^{\theta}_{h+1,h+1}, \sgn(\md))]  = E_h [(\md + E^{\theta}_{h, h} - E^{\theta}_{h+1,h+1}, \sgn(\md))],
 {\rm if} \ h\neq  n, \nonumber \\
\label{iii}
&& F_nE_n \md - \md' F_n E_n \md = [\lambda_n]_v \md, \quad \quad \quad \hspace{58pt}   {\rm if} \ \lambda=\lambda', \epsilon = - \epsilon' \neq 0,  \\
&&  \md' F_n E_n \md =0, \quad  \md F_n E_n  \md' =0, \quad  \quad  \hspace{56pt} \  {\rm if} \ \sgn (\md) =0, \sgn (\md') =\pm, \nonumber \\
\label{iv}\
&& (E_iF_j-F_jE_i) \md =0, \quad \hspace{114pt} \ {\rm if} \  i\neq j,\\
&& (E_iF_i-F_iE_i) \md  = [\lambda_{i+1}- \lambda_i]_v \md, \quad  \hspace{64pt} {\rm if}\ i\neq n, \nonumber \\
\label{v}\
&& (E_iE_iE_j- [2]_v  E_iE_jE_i+E_jE_iE_i) \md =0,\quad  \hspace{24pt}{\rm if}\ |i-j|=1,\\
&& (F_iF_iF_j-  [2]_v F_iF_jF_i+F_jF_iF_i) \md =0,\quad \quad  \hspace{23pt}{\rm if}\ |i-j|=1, \nonumber \\
\label{vi}\
&& (E_iE_j-E_jE_i) \md = 0, \quad (F_iF_j-F_jF_i)\md =0, \quad  \  {\rm if}\ |i-j|>1,\\
\label{vii}\
&&(E^2_nF_n+F_nE_n^2) \D =  [2]_v (E_nF_nE_n-E_n (v^{\lambda_{n+1} - \lambda_n+1} +v^{- \lambda_{n+1} +\lambda_n-1} )) \D,\\
&& (F_n^2E_n+E_nF_n^2) \D = [2]_v (F_nE_nF_n- (v^{\lambda_{n+1}-\lambda_n-2} +v^{-\lambda_{n+1} + \lambda_n+2} ) F_n) \D. \nonumber
\end{eqnarray}
\end{prop}

\begin{proof}
The proof of the identities in (\ref{i}) and (\ref{ii}) are straightforward.
We now show (\ref{iii}).
By  the multiplication formula (\ref{eq58}),  we have
\begin{equation*}
\begin{split}
F_n E_n \md
 &=[\lambda_n]_v [\md]+[( D_{\lambda'}+E_{n,n+2}^{\theta}, - \epsilon)],\\
\D' F_nE_n \D  &=[(D_{\lambda'} +E_{n,n+2}^{\theta}, -\epsilon)], \quad \quad \quad \mbox{where}\ \lambda_i'=\lambda_i-\delta_{i,n}-\delta_{i, n+2}.
\end{split}
\end{equation*}
So the first identity in (\ref{iii}) holds.
Observe that  if $\sgn(\D)=0$, then $\lambda_{n+1}\neq 0$; and if $\sgn (\D)\neq 0$ then $\lambda_{n+1}=0$.
We have the second  identities in (\ref{iii}) by this observation.

For the remaining relations, they can be proved by the following principle.
Suppose that $x\md = \sum C_{x\md, \A} \A$ with $C_{x\md, \A}\in \mcal A$.
We can  pick a large enough $p$ such that $_p\md$ and $_p\A$ all have non-negative entries.
For an appropriate $D'$, we have an element in $\mcal S_{D'}$ of the form $x\, {}_p\md$ defined in a similar way as that  in $\mcal K$.
We can write
\[
x \, {}_p \md = \sum {}_pC_{x {}\md, {}\A} (v, v')|_{v'=v^{-p}}  \ {}_p\A \quad \mbox{in} \ \mcal S_{D'},
\]
where  ${}_pC_{x {}\md, {}\A} (v, v') \in \mathfrak R$ in (\ref{R}).
If $x$ is of the form in the remaining relations, we have
\[
C_{x\md, \A} = {}_pC_{x {}\md, {}\A} (v, v')|_{v'=1}.
\]
This follows from the comparison of (\ref{eq58}), (\ref{eq57}) and (\ref{eq56}) in $\mcal K$ with
(\ref{eq22}), (\ref{eq24}) and (\ref{eq27})   in $\mcal S_{D'}$, respectively.
Now the remaining relations all hold in $\mcal S_{D'}$ for all $D'$ large enough by Proposition ~\ref{prop3},
so are those relations without specializing $v'$. Now relations  in $\mcal K$ are obtained by specializing $v'=1$.
\end{proof}

\subsection{The algebra $\mcal U$}\label{sec6}

In this section, we shall define a new algebra $\mcal U$ in the  completion of $\mcal K$ similar to  ~\cite[Section 5]{BLM90}.
We show that $\mcal U$ is a quotient of the algebra $\U$ defined in Section ~\ref{U-first}.

Let $\hat{\mcal K}$ be the $\mbb Q(v)$-vector space of all formal sum
$\sum_{\A\in \tilde{\Xi}_{\mbf D}}\xi_{\A} [\A]$ with $\xi_{\A}\in \mbb Q(v)$ and  a locally finite property, i.e.,
for any ${\mbf t}\in \mbb Z^N$, the sets $\{\A\in \tilde{\Xi}_{\mbf D}|{\rm ro}(\A)={\mbf t}, \xi_{\A} \neq 0\}$
and
$\{\A\in \widetilde{\Xi}_{\mbf D} | {\rm co}(\A)={\mbf t}, \xi_{\A} \neq 0\}$ are finite.
The space $\hat{\mcal K}$ becomes an  associative algebra over $\mbb Q(v)$
 when equipped  with  the following multiplication:
$$
\sum_{\A\in \tilde{\Xi}_{\mbf D}} \xi_{\A} [\A]   \cdot \sum_{\B \in \tilde{\Xi}_{\mbf D}} \xi_{\B} [\B]
=\sum_{\A, \B} \xi_{\A} \xi_{\B} [\A] \cdot [\B],
$$
where the product $[\A] \cdot [\B]$ is taken  in $\mcal K$.
This is shown in exactly the same as ~\cite[Section 5]{BLM90}.

Observe that the algebra $\hat{\mcal K}$ has a unit element $\sum\md$, the  summation  of  all diagonal signed matrices.

We define the following  elements in $\hat{\mcal K}$.
For any nonzero signed matrix $\A=(A, \epsilon)\in \wt{\Xi}_{\mbf D}$,
let $\hat{\A}=(\hat A, \epsilon)$,
where $\hat A$ is the matrix obtained
by replacing diagonal entries of $A$ by zeroes.
We set
$$
\hat{\Xi}_{\mbf D}= \{ \hat{\A} | \A\in \wt{\Xi}_{\mbf D} \}.
$$
For any $\hat{\A}$ in $\hat{\Xi}_{\mbf D}$ and ${\mbf j}=(j_1,\cdots, j_N)\in \mbb Z^N$, we define
\begin{equation} \label{wtA}
\hat{\A} ({\mbf j})=\sum_{\lambda}v^{\lambda_1j_1+\cdots+\lambda_{n+1}j_{n+1}}[ (\hat{\A} + D_{\lambda}, \sgn(\hat{\A})) ]\quad
\end{equation}
where the  sum runs through all $\lambda=(\lambda_i)\in \mbb Z^N$ such that
$(\hat{\A} + D_{\lambda}, \sgn(\hat{\A})) \in \wt{\Xi}_{\mbf D}$.

For any $i\in [1,n]$, there exist $\A=(A, \epsilon)$ such that
$\hat{\A}=(E_{i+1, i}^{\theta}, \epsilon)$ (resp. $\hat{\A}=(E_{i, i+1}, \epsilon)$.
So by (\ref{wtA}), the elements  $E_{i+1,i}^{\theta, \epsilon}(\mbf j)$ (resp. $E_{i, i+1}^{\theta, \epsilon}(\mbf j)$) are well-defined,
for any $\mbf j\in \mbb Z^N$.
Moreover, this definition is independent of the choice of $\hat{\A}$.
For $i\in [1,n]$, let
\begin{equation*}
E_i=E_{i+1,i}^{\theta, +}(0)+E_{i+1,i}^{\theta, 0}(0)+E_{i+1,i}^{\theta, -}(0)\quad{\rm and}\quad
F_i=E_{i, i+1}^{\theta,+}(0)+E_{i, i+1}^{\theta,0}(0)+E_{i, i+1}^{\theta,-}(0).
\end{equation*}
For simplicity, we shall write $E_i^{\epsilon}(\mbf j)$ (resp. $F_i^{\epsilon}(\mbf j)$) instead of  $E_{i+1,i}^{\theta, \epsilon}(\mbf j)$ (resp. $E_{i,i+1}^{\theta, \epsilon}(\mbf j)$).

We also define
\begin{equation*}
\begin{split}
  0(\mbf j)&=0^+(\mbf j)+0^0(\mbf j)+0^-(\mbf j),\quad {\rm where}\\
  0^{\epsilon}(\mbf j)&=\sum v^{\lambda_1j_1+\cdots+\lambda_{n+1}j_{n+1}}[\D],
\end{split}
\end{equation*}
where the sum runs through all diagonal matrices $\D$ with sign $\epsilon$ and $\lambda_i$'s are diagonal entries of $\D$.

Let $\mcal U$ be the subalgebra of $\hat{\mcal K}$ generated by $E_i, F_i, 0(\mbf j), 0^+(0), 0^0(0)$ and
$0^-(0)$ for all $i\in [1,n]$ and $\mbf j\in \mbb Z^N$.

\begin{prop}\label{prop-a}
The following relations hold in $\mcal U$.
 \allowdisplaybreaks
\begin{eqnarray} \label{eq60}
&&0(\mbf j)0(\mbf j')=0(\mbf j')0(\mbf j),\ 0^{\pm}(0)0(\mbf j)=0(\mbf j)0^{\pm}(0),\ 0^{\pm}(0)^2=0^{\pm}(0),\\
&& 0^+(0)+0^0(0)+0^-(0)=1,\  0^{\alpha}(0)0^{\epsilon}(0)=\delta_{\alpha,\epsilon}0^{\alpha}(0), \nonumber \\
\label{A-ii}
&&0(\mbf j)F_h=v^{j_h-j_{h+1}-\delta_{hn}j_{n+1}}F_h0(\mbf j),\
       0(\mbf j)E_h=v^{-j_h+j_{h+1}+\delta_{hn}j_{n+1}}E_h0(\mbf j),\\
 \label{A-iii}
&&0^{\pm}(0)E_h=(1-\delta_{hn})E_h 0^{\pm}(0),\  F_h0^{\pm}(0)=(1-\delta_{hn})0^{\pm}(0)F_h,\\
\label{A-iv}
&& 0^{\pm}(0)F_nE_n-F_nE_n0^{\mp}(0)=\frac{0(\underline n)-0(-\underline n)}{v-v^{-1}}(0^{\pm}(0)-0^{\mp}(0)),\\
\label{A-v}&&F_hE_h-E_hF_h=(v-v^{-1})^{-1}(0(\underline h-\underline{h+1})-0(\underline{h+1} -\underline h)),\\
&&E_iF_n=F_nE_i,\quad F_iE_n=E_nF_i,\ \hspace{.7cm}  {\rm if} \ i\in [1,n-2], \nonumber \\
\label{A-vi} &&E_iE_j=E_jE_i,\quad F_iF_j=F_jF_i,\quad \hspace{.6cm}  {\rm if}\ |i-j|>1,\\
\label{A-vii}
&& E^2_nF_n+F_nE_n^2= [2]_v (E_nF_nE_n
      -E_n(v0(\underline{n+1}-\underline n)+v^{-1}0(\underline n- \underline{n+1}))),\\
&&F_n^2E_n+E_nF_n^2=[2]_v (F_nE_nF_n
      -(v0(\underline{n+1}-\underline n))+v^{-1}0(\underline n-\underline{n+1}))F_n), \nonumber \\
\label{A-viii} &&E_i^2E_j- [2]_v E_iE_jE_i+E_jE_i^2=0,\quad\ {\rm if}\ |i-j|=1,\ i,j\in [1,n-1],\\
&&F_i^2F_j-[2]_v F_iF_jF_i+F_jF_i^2=0,\quad \hspace{.4cm} {\rm if}\ |i-j|=1,\ i,j\in [1,n-1], \nonumber
\end{eqnarray}
where $\mbf j, \mbf j'\in \mbb Z^N$, $\alpha$, $\epsilon\in \{ \pm , 0\}$, $h, i, j\in [1, n]$ and  $\underline i \in \mbb N^N$ is the vector whose $i$-th entry is 1 and 0 elsewhere.
\end{prop}

\begin{proof}
We show (\ref{A-ii}).
By checking the values of functions $s_l$ and $s_r$ defined in (\ref{eq68}) at $ 0(\mbf j)$ and $F_n$, we have
\begin{equation*}
\label{eq64}
\begin{split}
0(\mbf j)F_n
&=0^+(\mbf j)F^+_n(0)+0^0(\mbf j)F^0_n(0)+0^-(\mbf j)F^-_n(0)\\
&=\textstyle\sum_{\lambda,\lambda'}v^{\sum \lambda_kj_k}[D_{\lambda}^+]
[( E_{n,n+1}^{\theta,+}+D_{\lambda'}, +)]\\
& +\textstyle\sum_{\lambda,\lambda'}v^{\sum \lambda_kj_k}[D_{\lambda}^-]
[( E_{n,n+1}^{\theta,-}+D_{\lambda'}, -)]
\textstyle+\sum_{\lambda,\lambda'}v^{\sum \lambda_kj_k}[D_{\lambda}^0]
[(E_{n,n+1}^{\theta,0}+D_{\lambda'}, 0)]\\
&\textstyle =v^{j_n}(F_n^+(\mbf j)+F_n^0(\mbf j)+F_n^-(\mbf j))=v^{j_n}F_n(\mbf j),
\end{split}
\end{equation*}
where the sums run through in an obvious range by the definition in (\ref{wtA}).
\begin{equation*}
\begin{split}
F_n 0(\mbf j)
&=F^+_n(0)0^0(\mbf j)+F^0_n(0)0^0(\mbf j)+F^-_n(0)0^0(\mbf j)\\
&=\textstyle \sum_{\lambda, \lambda' }v^{\sum \lambda_kj_k}[(E_{n,n+1}^{\theta,+}+D_{\lambda'}, +) ][D_{\lambda}^0]\\
&\textstyle +\sum_{\lambda,\lambda'}v^{\sum z_kj_k} [(E_{n,n+1}^{\theta,0}+D_{\lambda}, 0)] [D_{\lambda}^0]
+\sum_{\lambda, \lambda'}v^{\sum \lambda_kj_k}[(E_{n,n+1}^{\theta,-}+D_{\lambda'}, -)][D_{\lambda}^0]\\
 &=v^{2j_{n+1}}(F_n^+(\mbf j)+F_n^0(\mbf j)+F_n^-(\mbf j))=v^{2j_{n+1}}F_n(\mbf j).
 \end{split}
 \end{equation*}
 So we have the first identity in (\ref{A-ii})  for the case of $h=n$.
 Other cases for the first identity and all other identities in (\ref{eq60}) and (\ref{A-ii}) can be shown similarly.

We show (\ref{A-iii}).
By the definition of $0^+(0)$ and $F_h$ for $h<n$, we have
\begin{equation*}
\begin{split}
0^+(0)F_h
&=0^+(0)F_h^+(0)
= \textstyle\sum_{\lambda,\lambda'}[D_{\lambda}^+]
[(D_{\lambda'}^+ + E_{h,h+1}^{\theta}, +)]\\
&\textstyle=\sum_{\lambda'} [( D_{\lambda'}^+ + E_{h+1,h}^{\theta}, +)]=F_h0^+(0).
\end{split}
\end{equation*}
The other identities in (\ref{A-iii}) can be shown similarly.

We show (\ref{A-iv}). By Proposition \ref{prop13} (\ref{A-iii}), we have
  \begin{equation*}
  \begin{split}
    &0^+(0)F_nE_n=
    \textstyle \sum_{\lambda}[D_{\lambda}^+]F_nE_n
    =\sum_{\lambda}([\lambda_n]_v [D^+_{\lambda_n}]+[(D_{\lambda}^++E_{n,n+2}^{\theta}, +)])\\
   &F_nE_n0^-(0)=\textstyle \sum_{\lambda}F_nE_n[D_{\lambda}^-]
    =\sum_{\lambda}([\lambda_n]_v [D^-_{\lambda_n}]+[(D_{\lambda}^++E_{n,n+2}^{\theta}, +)]).
  \end{split}
  \end{equation*}
Therefore,
\begin{equation*}
0^+(0)F_nE_n-F_nE_n0^-(0)\textstyle=\sum_{\lambda}
[\lambda_n]_v
([D^+_{\lambda}]-[D^-_{\lambda}])
    =\frac{0(\underline n)-0(- \underline n)}{v-v^{-1}}(0^{+}(0)-0^{-}(0)).
\end{equation*}

We now show (\ref{A-vii}).
By definition, we have
\begin{equation*}
v0(\underline{n+1}- \underline n)F_n=F_n  0 (\underline{n+1}- \underline n)=
\textstyle  \sum_{\lambda}v^{\lambda_{n+1}-\lambda_n-2}F_nD_{\lambda}^0.
\end{equation*}
Similarly, $v^{-1}0(\underline n- \underline{n+1}) F_n= \sum_{\lambda}v^{\lambda_n-\lambda_{n+1}+2}F_nD_{\lambda}^0$.
Moreover,
\begin{equation*}
\textstyle F_n^2E_n+E_nF_n^2-[2]_v F_nE_nF_n=\sum_{\lambda} (F_n^2E_n+E_nF_n^2-[2]_v F_nE_nF_n)D_{\lambda}^0.
\end{equation*}
The identity (\ref{A-vii}) follows from Proposition \ref{prop13}.
All other identities in (\ref{A-v})-(\ref{A-viii}) can be shown similarly.
\end{proof}

\begin{prop} \label{Upsilon}
The assignment $E_i\mapsto E_i$, $F_i\mapsto F_i$, $H_a \mapsto 0(- \underline a)$
 and $J_{\alpha}\mapsto 0^{\alpha}(0)$,  for any $i\in [1,n]$, $a\in [1, n+1]$ and $\alpha\in \{0, +, -\}$,
 defines a surjective  algebra homomorphism $\Upsilon: \U\rightarrow \mcal U$ where $\U$ is defined in Section ~\ref{U-first}.
\end{prop}

\begin{proof}
Under the map $\Upsilon$, all defining relations of $\U$ map  to the  corresponding relations in $\mcal U$ given in Proposition \ref{prop-a} except the commutator relation between $J_{\pm}$ and $F_nE_n$.
Since
\begin{equation*}
\textstyle 0(\underline{n+1}-\underline n)0^{\pm}(0)=\sum_{\lambda}v^{\lambda_{n+1}-\lambda_n}[D_{\lambda}^{\pm}]
=\sum_{\lambda}v^{-\lambda_n}[D_{\lambda}^{\pm}]=0(- \underline n)0^{\pm}(0),
\end{equation*}
we have
\begin{equation*}
\begin{split}
&\Upsilon(J_{\pm}F_nE_n-F_nE_nJ_{\mp}-\frac{H_n^{-1}H_{n+1} -H_n H_{n+1}^{-1} }{v-v^{-1}} (J_{\pm}-J_{\mp}))\\
& =0^{\pm}(0)F_nE_n-F_nE_n0^{\mp}(0)-\frac{0( \underline n)-0(- \underline n)}{v-v^{-1}}(0^{\pm}(0)-0^{\mp}(0))=0
\end{split}
\end{equation*}
This shows that $\Upsilon$ is an algebra homomorphism. The surjectivity is clear.
\end{proof}

\begin{rem}
It is not clear if $\ker \Upsilon=0$.
\end{rem}

\section{Case II}

In this section, we turn to the case when  all flags at the $n$-th step are assumed  to be maximal isotropic.

\subsection{The second double centralizer}
\label{sec6.1}

We define $\mscr X^m$ to be the subset of $\mscr X$ in Section \ref{setup} subject to the condition that the $n$-th step of the flags is maximal isotropic. In particular, we have
$V_n = V_{n+1}$ for any $V\in \mscr X^m$, and thus
$$
\mscr X^m = \mscr X^2 \sqcup \mscr X^3.
$$
Similar to the definition of the algebra $\mcal S$ in Section \ref{sec4.2}, we consider the convolution algebra
$$\mcal S^m=\mcal A_G (\mscr X^m \times \mscr X^m)$$ on
$\mscr  X^m \times \mscr  X^m $ and the free $\mcal A$-module
$$\mcal W=\mcal A_G(\mscr  X^m \times \mscr Y).$$
Under the  convolution product, $\mcal W$ has a $\mcal S^m$-$\mcal H_{\mscr Y}$-bimodule structure.
By ~\cite{P09},  we have

\begin{lem} \label{Geometric-duality-II}
The triple $(\mcal S^m, \mcal H_{\mscr Y}; \mcal W)$ satisfies the  double centralizer property, i.e.,
\begin{align}
\label{Geometric-duality-2}
\End_{\mcal S^m}(\mcal W)=\mcal H_{\mscr Y}\quad {\rm and}\quad \End_{\mcal H_{\mscr Y}}(\mcal W)=\mcal S^m, \quad \mbox{if}\; n\geq d.
\end{align}
\end{lem}

Let $\Pi^m=\{B\in \Pi|b_{n+1,j}=0, \forall j\}$, where $\Pi$ is defined in Section \ref{sec3.2}.
A restriction of the bijection (\ref{eq44}) in Section \ref{sec3.2} yields a bijection
$$
G \backslash \mscr  X^m \times \mscr Y \xrightarrow{\sim} \Pi^m.
$$
Moreover, the isomorphism (\ref{V-V}) restricts to an isomorphism
\begin{align} \label{W-W}
\mbf W^{\otimes d} \overset{~}{\to} \mbb Q(v)\otimes \mcal W,
\end{align}
where $\mbf W^{\otimes d}$ is defined in Section ~\ref{Um}.


Observe that the algebra $\mcal S^m$ is naturally a subalgebra of $\mcal S$, while $\mcal W$ is an $\mcal A$-submodule of  $\mcal V$ in (\ref{V}).
So we can define the function $E_i$, $F_i$, $H_a^{\pm}$, for $i\in [1, n-1]$,  $a\in [1, n]$,  and $J_{\pm}$ in $\mcal S^m$  to be  the restrictions of the functions in $\mcal S$ under the same notations, respectively.
We further define
\begin{equation}
T(V,V')=
\left\{
\begin{array}{ll}
v^{1-\lambda'_n}, & {\rm if}\ |V_n\cap V_n'|= d-1,\ V_j=V_j',\  \forall j\in [1,n-1];\\
0,&{\rm otherwise}.
\end{array}
\right.
\end{equation}
By definitions, we have
\begin{equation} \label{T=FE}
T= \left ( F_nE_n+\frac{H_nH_{n+1}^{-1} - H_n^{-1}H_{n+1} }{v-v^{-1}} \right )|_{\mscr X^m \times \mscr X^m}.
\end{equation}

We see immediately

\begin{lem} \label{Sm-W}
The actions of $E_i$, $F_i$, $H_a^{\pm 1}$ and $J_{\pm}$ for $i\in [1, n-1]$, $a\in [1,n]$  on $\mcal W$ are given by the formulas in Corollary ~\ref{cor9}, with the action of $T$ on $\mcal W$ given by
$F_nE_n+\frac{H_nH_{n+1}^{-1} - H_n^{-1}H_{n+1} }{v-v^{-1}}$ from Corollary ~\ref{cor9} again.
\end{lem}

\subsection{Relations for $\mcal S^m$}

We now determine the relations for the algebra $\mcal S^m$. By using Proposition ~\ref{prop3} and (\ref{T=FE}),  we have

\begin{prop} \label{prop7}
The functions $E_i$, $F_i$ and $H_a^{\pm 1}$, for $i\in [1, n-1]$, $a\in [1,n]$  together with the functions $J_{\pm}$ and $T$ in $\mcal S^m$
satisfy the defining relations of the algebra $\U^m$ in Section \ref{Um}.
\end{prop}

\begin{rem} \label{Sm-action}
The function $T$ has a geometric interpretation.  More preciely, we  set
\begin{equation*}
  S(T)=\{(V,V')||V_n\cap V_n'|=d-1,\ V_j=V_j', \forall j\in [1,n-1]\}.
\end{equation*}
By (\ref{partial-order}), we see that  $S(T)$ is a smooth closed subvariety of $\mscr X^m\times \mscr X^m$  over the algebraic closure of the field $\mbb F_q$. So the function $T$ is the  function version of  the intersection complex associated to the variety $S(T)$, up to a shift.
\end{rem}

The rest of this subsection is devoted to give another more direct proof of  Proposition \ref{prop7}.

We first define an auxiliary  function   $\wt T$  by
\begin{equation}
\wt T(V,V')=\left\{\begin{array}{ll}
    v^{1-\lambda'_n}, & {\rm if}\ |V_n\cap V_n'|\geq d-1,\ V_j=V_j', \forall j\in [1,n-1];\\
    0,&{\rm otherwise},
  \end{array}\right.
\end{equation}
where $\lambda'_n=|V'_n|-|V'_{n-1}|$.
Moreover, we have
\begin{equation}\label{eq28}
\wt T=T+ vK_n = (F_nE_n+v\frac{vK_n-v^{-1}K_n^{-1}}{v-v^{-1}})|_{\mscr X^m \times \mscr X^m}.
\end{equation}

By a direct computation, we have
\begin{equation*}
  E^2_{n-1}\wt T(V,V')=\left\{\begin{array}{ll}
  (v^2+1)v^{-3\lambda_n'}& {\rm if}\ V_{n-1}\overset{2}{\subset}V_{n-1}'\subset V_n,\ |V_n\cap V_n'|\geq d-1,\vspace{6pt}\\
  0&{\rm otherwise}.
  \end{array}
  \right.
\end{equation*}
\begin{equation*}
\wt T E^2_{n-1}(V,V')=\left\{\begin{array}{ll}
  (v^2+1)v^{-3\lambda_n'-2}& {\rm if}\ V_{n-1}\overset{2}{\subset}V_{n-1}',\ |V_n\cap V_n'|\geq d-1,\vspace{6pt}\\
  0&{\rm otherwise}.
  \end{array}
  \right.
\end{equation*}
\begin{equation*}
E_{n-1}\wt T E_{n-1}(V,V')=\left\{\begin{array}{ll}
  (v^2+1)v^{-3\lambda_n'-1}& {\rm if}\ V_{n-1}\overset{2}{\subset}V_{n-1}'\subset V_n,\ |V_n\cap V_n'|\geq d-1,\vspace{6pt}\\
  v^{-3\lambda_n'-1}& {\rm if}\ V_{n-1}\overset{2}{\subset}V_{n-1}'\not \subset V_n,\ |V_n\cap V_n'|\geq d-1,\vspace{6pt}\\
  0&{\rm otherwise}.
  \end{array}
  \right.
\end{equation*}
So we have
$$
E_{n-1}^2\wt T-[2]_v E_{n-1}\wt TE_{n-1}+\wt TE_{n-1}^2=0,
$$
and by (\ref{eq28}) it implies that
\begin{align} \label{ET}
E_{n-1}^2 T-[2]_v E_{n-1} TE_{n-1}+TE_{n-1}^2=0.
\end{align}

A direct computation shows that we have
\begin{equation*}
\wt T^2E_{n-1}(V,V')=\left\{\begin{array}{ll}
  (\frac{v^{2\lambda'_n+2}-1}{v^2-1}+1)v^{-3\lambda'_n}& {\rm if}
  \ V_{n-1}\overset{1}{\subset}V'_{n-1},\ V_n=V'_n, \vspace{6pt}\\
  (v^2+1)v^{-3\lambda'_n}& {\rm if}\ V_{n-1}\overset{1}{\subset}V'_{n-1},\
   |V_n\cap V_n'|=d-2,\vspace{6pt}\\
  2v^{-3\lambda'_n}& {\rm if}\ V_{n-1}\overset{1}{\subset}V'_{n-1},\
   |V_n\cap V_n'|=d-1,\vspace{6pt}\\
  0&{\rm otherwise}.
  \end{array}
  \right.
\end{equation*}
\begin{equation*}
E_{n-1}\wt T^2(V,V')=\left\{\begin{array}{ll}
  (\frac{v^{2\lambda'_n}-1}{v^2-1}+1)v^{-3\lambda'_n+2}& {\rm if}
  \ V_{n-1}\overset{1}{\subset}V'_{n-1}\subset V_n,\ V_n=V'_n, \vspace{6pt}\\
  (v^2+1)v^{-3\lambda'_n+2}& {\rm if}\ V_{n-1}\overset{1}{\subset}V'_{n-1}\subset V_n,\
   |V_n\cap V_n'|=d-2,\vspace{6pt}\\
  2v^{-3\lambda'_n+2}& {\rm if}\ V_{n-1}\overset{1}{\subset}V'_{n-1}\subset V_n,\
   |V_n\cap V_n'|=d-1,\vspace{6pt}\\
  0&{\rm otherwise}.
  \end{array}
  \right.
\end{equation*}
\begin{equation*}
\wt TE_{n-1}\wt T(V,V')=\left\{\begin{array}{ll}
  (\frac{v^{2\lambda'_n}-1}{v^2-1}+1)v^{-3\lambda'_n+1}& {\rm if}
  \ V_{n-1}\overset{1}{\subset}V'_{n-1}\subset V_n,\ V_n=V'_n, \vspace{6pt}\\
  (v^2+1)v^{-3\lambda'_n+1}& {\rm if}\ V_{n-1}\overset{1}{\subset}V'_{n-1}\subset V_n,\
   |V_n\cap V_n'|=d-2,\vspace{6pt}\\
v^{-3\lambda'_n+1}& {\rm if}\ V_{n-1}\overset{1}{\subset}V'_{n-1}\not \subset V_n,\
   |V_n\cap V_n'|=d-2,\vspace{6pt}\\
  2v^{-3\lambda'_n+1}& {\rm if}\ V_{n-1}\overset{1}{\subset}V'_{n-1}\subset V_n,\
   |V_n\cap V_n'|=d-1,\vspace{6pt}\\
   v^{-3\lambda'_n+1}& {\rm if}\ V_{n-1}\overset{1}{\subset}V'_{n-1}\not \subset V_n,\
   |V_n\cap V_n'|=d-1,\vspace{6pt}\\
  0&{\rm otherwise}.
  \end{array}
  \right.
\end{equation*}
\begin{equation*}
E_{n-1}\wt T(V,V')=\left\{\begin{array}{ll}
  v^{-2\lambda'_n+1}& {\rm if}\ V_{n-1}\overset{1}{\subset}V'_{n-1}\subset V_n,\
   |V_n\cap V_n'|\geq d-1,\vspace{6pt}\\
  0&{\rm otherwise}.
  \end{array}
  \right.
\end{equation*}
\begin{equation*}
\wt T E_{n-1}(V,V')=\left\{\begin{array}{ll}
  v^{-2\lambda'_n}& {\rm if}\ V_{n-1}\overset{1}{\subset}V'_{n-1},\
   |V_n\cap V_n'|\geq d-1,\vspace{6pt}\\
  0&{\rm otherwise}.
  \end{array}
  \right.
\end{equation*}
This implies that we have
$$
\wt T^2E_{n-1}-[2]_v\wt TE_{n-1}\wt T+E_{n-1}\wt T^2=E_{n-1}-(v-v^{-1})(v\wt TE_{n-1}-E_{n-1}\wt T)H_nH_{n+1}^{-1},
$$
which implies again by (\ref{eq28}) that
\begin{equation}\label{TE}
T^2E_{n-1}- [2]_v TE_{n-1}T+E_{n-1}T^2=E_{n-1}.
\end{equation}

Now apply the map $\rho$ in (\ref{rho})  to (\ref{ET}) and (\ref{TE}), we get
\begin{align*}
& F_{n-1}^2 T - [2]_v F_{n-1}TF_{n-1}+ TF_{n-1}^2=0,\\
& T^2F_{n-1} - [2]_vTF_{n-1}T+F_{n-1}T^2=F_{n-1}.
\end{align*}
The other defining equations of $\U^m$ are straightforward to check and skipped. This finishes the proof of  Proposition \ref{prop7}.

\subsection{Generators and bases for $\mcal S^m$}

We consider the following subset of $\Xi_{\mbf D}$.
$$\Xi_{\mbf D}'=\{ \A \in \Xi_{\mbf D}| \ro (\A)_{n+1}=\co(\A)_{n+1}=0\}.$$
We then have $\sgn(\A) \in \{ +, -\}$ if $\A\in \Xi_{\mbf D}'$. Moreover, we have a bijection
\begin{align}
G\backslash \mscr X^m\times \mscr X^m \simeq \Xi_{\mbf D}',
\end{align}
inherited from the bijection (\ref{Phi}).

Recall from Theorem ~\ref{thm1}  that we set $R_{ij}=\sum_{k=1}^ia_{kj}$ for a signed matrix $\A=(A, \epsilon)$.
Let $\mathfrak e_{i,t}$ denote a signed matrix such that
$\mathfrak e_{i, t} -R_{i, i+t} E^{\theta}_{i, i+1}$ is diagonal.
For  a sequence $a_s, a_{s+1},\cdots, a_{r}$ with $s\leq r$, we set
\[
\overset{s}{\underset{i=r}{\sqcap}}
a_i=a_ra_{r-1}\cdots a_s.
\]

\begin{thm}\label{thm2}
For any $\A=A^{\epsilon} \in \Xi'_{\mbf D}$,
there exists a product of  signed  matrices   $\mathfrak e_{i, t}$
\begin{equation} \label{Ma-ii}
\begin{split}
\mathfrak n_{\A} = \left (
\overset{n+1}{\underset{t=N-1}{\sqcap}}
\overset{1}{\underset{i=N-t}{\sqcap}}
[\mathfrak e_{i, t}]
\right )
\overset{1}{\underset{t=n}{\sqcap}}
 \left (
\overset{n+2}{\underset{i=N-t}{\sqcap}}
[\mathfrak e_{i, t}]
([\mathfrak e_{n, t+1}][\mathfrak e_{n+1, t}])
\overset{n-t+1}{\underset{i=n-1}{\sqcap}}
[\mathfrak e_{i, t+1}]
\overset{1}{\underset{i=n-t}{\sqcap}}
[\mathfrak e_{i,t}]
 \right )
\end{split}
\end{equation}
 such that
\begin{equation}
\mathfrak n_{\A} =[\A]+{\rm lower\ terms},
\end{equation}
where the matrices $\mathfrak e_{i, t}$  are completely determined by the conditions
$\ro ( \mathfrak e_{1, N-1})=\ro ( \A)$ and $\co (\mathfrak e_{1, 1}) = \co (\A)$ and the signs of  $\mathfrak e_{i, t}$ are inductively determined by the conditions that
$s_r(\mathfrak e_{1,1}) = s_r (\A) $ and
$s_l (\mathfrak e_{i, t}) = s_r (\mathfrak e_{i, t} ) + (-1)^{ s_r (\mathfrak e_{i, t} ) } \p ( \mathfrak{ e}_{i , t} )$.
\end{thm}

\begin{proof}
The proof is  a modification of the one of  Theorem \ref{thm1}.
We show it for $n=2$.
We consider a signed matrix $\A=(A, +1)$ in $\Xi'_{\mbf D}$.
Without lost of generality, we assume that $\ur(\A)$ is even, i.e. $a_{14}+a_{24}+a_{15}+a_{25}$ is even.
Let $B_{10}$ be a diagonal matrix with diagonal entries being the entries of $\co (\A)$.
Let $B_{11}$ be the matrix such that $B_{11}-R_{12}E^{\theta}_{12}$ is a diagonal matrix and ${\rm co}(B_{11})={\rm ro}(B_{10})$.
Let $[M_1]=[D_{24}+R_{24}E^{\theta}_{23}][D_{34}+R_{34}E^{\theta}_{34}]$, where $D_{24}$
and $D_{34}$ are uniquely determined  by the condition  ${\rm co}(D_{34}+R_{34}E^{\theta}_{34})={\rm ro}(B_{11})$ and
${\rm co}(D_{24}+R_{24}E^{\theta}_{23})={\rm ro}(D_{34}+R_{34}E^{\theta}_{34})$.
Let $B_{14}$ be the matrix such that $B_{14}-R_{45}E^{\theta}_{45}$ is a diagonal matrix and ${\rm co}(B_{14})={\rm ro}(D_{24}+R_{24}E^{\theta}_{23})$.

By the example in Remark \ref{remEx},
$[M_1]$ is either $[(D_{\lambda},\epsilon)]$ or $[(D_{\lambda'}+R_{34}E^{\theta}_{24}, \epsilon)]$
up to a scalar for some $\lambda, \lambda'$ and $\epsilon$.
So we can talk about the sign of $M_1$.
We set
\begin{align*}
\B_{10} & = (B_{10}, + ), \quad
\B_{11}  = (B_{11}, +), \\
\M_1 & = \begin{cases}
  (M_1, +),  & {\rm if}\  a_{14}+a_{24}\ {\rm is\ even},\\
  (M_1, -),& {\rm if}\  a_{14}+a_{24}\ {\rm is\ odd},
\end{cases}\\
\B_{14} & = \begin{cases}
  (B_{14}, +),& {\rm if}\  a_{14}+a_{24}\ {\rm is\ even},\\
  (B_{14}, -),& {\rm if}\  a_{14}+a_{24}\ {\rm is\ odd}.
\end{cases}
\end{align*}
By Corollary \ref{cor7}, we have
$$[\B_{14}] [\M_1] [\B_{11}] [\B_{10}]=[\A_1]+ {\rm lower\ terms},$$
where
$$
\A_1= \begin{cases}
  (A_1, +),& {\rm if}\  a_{14}+a_{24}\ {\rm is\ even},\\
  (A_1, -),& {\rm if}\  a_{14}+a_{24}\ {\rm is\ odd},
\end{cases}
\quad
A_1=
\begin{pmatrix}
*&R_{12}&0&0&0\\R_{45}&*&0&R_{24}&0\\
0&0&0&0&0\\0&R_{24}&0&*&R_{45}\\0&0&0&R_{12}&*
\end{pmatrix},
$$
and the $*$'s are some positive numbers unique determined  by ${\rm co}(A_1)={\rm co}(B_{10})$.
Now let $B_{31}$ be the matrix such that $B_{31}-R_{14}E^{\theta}_{12}$ is a diagonal matrix and ${\rm co}(B_{31})={\rm ro}(B_{14})$ and
$[M_2]=[D_{25}+R_{25}E^{\theta}_{23}][D_{35}+R_{35}E^{\theta}_{34}]$.
We set
\begin{align*}
\B_{31} & = \begin{cases}
  (B_{31}, +),& {\rm if}\  a_{14}+a_{24}\ {\rm is\ even},\\
  (B_{31}, -),& {\rm if}\  a_{14}+a_{24}\ {\rm is\ odd},
\end{cases}\\
\M_2 &= ( M_2, +).
\end{align*}
Then we have
$$[\M_2] [\B_{31}] [\A_1]=[\A_2]+{\rm lower\ terms},$$
where  $\A_2=(A_2, +)$ with
$$
A_2=
\begin{pmatrix}
*&R_{12}&0&a_{14}&0\\a_{45}&*&0&a_{24}&R_{25}\\
0&0&0&0&0\\R_{25}&a_{24}&0&*&a_{45}\\0&a_{14}&0&R_{12}&*
\end{pmatrix},
$$
$D_{ij}$ and the $*$'s are unique determined .
Let $B_{41}$  be the matrix such that $B_{41}-R_{15}E^{\theta}_{12}$ is a diagonal matrix and ${\rm co}(B_{41})={\rm ro}(A_{2})$. We set
\begin{align*}
\B_{41} = (B_{41}, +).
\end{align*}
We have
$$[\B_{41}] [\M_2] [\B_{31}] [\B_{14}] [\M_1] [\B_{11}] [\B_{10}]=[\A]+{\rm lower\ terms}.$$
This finishes the proof for $n=2$ and positively signed matrices.
The case for  the negatively signed matrices  can be shown similarly and so is the  general case.
\end{proof}

By Theorem ~\ref{thm2}, we can deduce the following results for $\mcal S^m$ similar to those for $\mcal S$.

\begin{prop} \label{lumpsum}

(a)
The algebra $\mcal S^m$ is generated by $[\mathfrak e]$ such that either  $\mathfrak e - R E^{\theta}_{n, n+2}$, $\mathfrak e-R E^{\theta}_{i, i+1}$  or $\mathfrak e - R E^{\theta}_{i+1,i}$ is diagonal for some $R\in \mbb N$ and $i\in [1, n-1]$.

(b)
The algebra $\mcal S^m$ admits a standard basis $\{ [\A]|\A\in \Xi'_{\mbf D}\}$,
a monomial basis $\{ \mathfrak n_{\A} | \A\in \Xi'_{\mbf D}\}$ and the  canonical basis
$\{\{ \A\} | \A\in \Xi'_{\mbf D}\}$, where $\mathfrak n_{\A}$ is in (\ref{Ma-ii}).

 (c)
The algebra $\mbb Q(v)\otimes_{\mcal A} \mcal S^m$ is generated by the functions $E_i, F_i$, $H_a^{\pm 1}$, $J_{\pm}$ and $T$ for  any $i\in [1, n-1]$ and $a\in [1, n]$.
\end{prop}

\begin{rem}
The basis $\{\mathfrak n_{\A} \}$ of $\mcal S^m$ is not a subset of the basis $\{ \m_{\A}\}$ of $\mcal S$ in Corollary \ref{S-M}.
But the canonical basis $\{ \{\A\} | \A \in \Xi'_{\mbf D}\}$ of $\mcal S^m$  is a subset of the basis $\{ \{ \A \} | \A \in \Xi_{\mbf D} \}$ of $\mcal S$. 
\end{rem}

\subsection{The algebra $\mcal K^m$}

Recall that $\Xi_{\mbf D}'=\{\A\in \Xi_{\mbf D}|\ro(\A)_{n+1} = \co (\A)_{n+1} =0 \}$.
Let
$$
\widetilde{\Xi}'_{\mbf D}=\{ \A\in \widetilde{\Xi}_{\mbf D}|a_{n+1,j}=a_{j,n+1}=0, \forall j\}.
$$
Let $\mcal K^m$ be the subalgebra of $\mcal K$ spanned by the elements $[\A]$ for any $\A\in \wt{\Xi}'_{\mbf D}$.
Notice that $\mcal K^m$ can be   obtained via a stabilization similar to Section ~\ref{Stab} by using the algebras $\mcal S^m$.
Similar to Theorem ~\ref{thm2}, we have
\[
\mathfrak n_{\A} = \A +\mbox{lower terms}, \quad \forall \A\in \wt{\Xi}'_{\mbf D},
\]
where $\mathfrak n_{\A}$ is defined in (\ref{Ma-ii}).
Moreover, by (\ref{eq58}),
we have
\[
[(D_{\lambda} + E_{n,n+1}^{\theta}, \epsilon )]\cdot
[(D_{\lambda} +E_{n+1,n}^{\theta}, \epsilon')] =
 \begin{cases}
[(D_{\lambda}+ E^{\theta}_{n,n}, \epsilon)] & \mbox{if} \ \epsilon= \epsilon',\\
[(D_{\lambda} + E^{\theta}_{n, n+2}, \epsilon)] & \mbox{if} \ \epsilon \neq \epsilon'.
\end{cases}
\]
From this observation, we have the following results for $\mcal K^m$ and $\mbb Q(v)\otimes_{\mcal A} \mcal K^m$ similar to those for $\mathcal K$ and $\mbb Q(v) \otimes_{\mcal A} \mathcal K$.

\begin{prop} \label{K-generator}
(a) The algebra $\mcal K^m$ is generated by the elements $[\mathfrak e]$ such that either $\mathfrak e - R E^{\theta}_{n, n+2}$,
$\mathfrak e- R E^{\theta}_{i, i+1}$ or $\mathfrak e - R E^{\theta}_{i+1, i}$
 is diagonal for some $i\in [1, n-1]$ and $R\in \mbb N$.

(b) The algebra $\mbb Q(v)\otimes_{\mcal A} \mcal K^m$ is generated by the elements
$[\mathfrak e]$ such that
either $\mathfrak e$, $\mathfrak e - E^{\theta}_{n, n+2}$,  $\mathfrak e -E^{\theta}_{i,i+1}$ or $\mathfrak e - E^{\theta}_{i+1,i}$ is diagonal for some
$i \in [1,n-1]$.

(c)
The algebras $\mcal K^m$ and $\mbb Q(v)\otimes_{\mcal A} \mcal K^m$ possess three bases: the standard basis $\{ [\A] | \A\in \wt{\Xi}'_{\mbf D}\}$, the monomial basis
$\{ \mathfrak n_{\A} | \A\in \wt{\Xi}'_{\mbf D}\}$,  and the canonical basis $\{ \{ A\} | \A\in \wt{\Xi}'_{\mbf D}\}$.

(d)
The restriction of $\Psi$ in Theorem ~\ref{KS} defines a surjective algebra homomorphism $\Psi^m: \mcal K^m \to \mcal S^m$ such that
$\Psi^m ([\A]) = [\A]$ if $\A\in \Xi_{\mbf D}'$ and $0$ otherwise. Moreover $\Psi^m (\{ \A\})=\{\A\}$ if $\A\in \Xi'_{\mbf D}$ and $0$ otherwise.
\end{prop}

\subsection{A presentation of $\mbb Q(v)\otimes_{\mcal A} \mcal K^m$}


To a diagonal signed matrix $\D =(D_{\lambda}, \epsilon)$ in $\wt{\Xi}'_{\mbf D}$, we  set
\begin{equation}\label{eq65}
 T \D= F_n E_n \D - [\lambda_n]_v  \D.
 \end{equation}
where $F_nE_n \D$ is defined in (\ref{eq63}) and  lies in $\mcal K^m$. Note that  $\lambda_{n+1}=0$ in this case.

\begin{prop}\label{prop14}
Let $\D=(D_{\lambda}, \epsilon)$ and $\D' = (D_{\lambda'}, \epsilon')$ be two signed diagonal matrices in $\wt{\Xi}'_{\mbf D}$.
The following relations hold in $\mcal K^m$.
\begin{eqnarray*}
&& \D  \D'  = \delta_{\D, \D' } \D,\\
&& T \D - \D'  T \D=0, \quad \hspace{5cm} {\rm if} \ \lambda=\lambda', \epsilon =\epsilon',\\
&& (E_iF_j-F_jE_i) \D =0,\quad \hspace{4.1cm} {\rm if}\ i\neq j,\\
&& (E_iF_i-F_iE_i) \D = [\lambda_{i+1}-\lambda_{i}]_v \D, \quad \hspace{2.15cm}  {\rm if}\ i\neq n,\\
&& (E_{n-1}^2T-[2]_v E_{n-1}TE_{n-1}+TE_{n-1}^2) \D =0,\\
&& (F_{n-1}^2T-[2]_v F_{n-1}TF_{n-1}+TF_{n-1}^2) \D=0, \\
&& (T^2E_{n-1}-[2]_v TE_{n-1}T+E_{n-1}T^2) \D =E_{n-1} \D,\\
&& (T^2F_{n-1}-[2]_v TF_{n-1}T+F_{n-1}T^2) \D= F_{n-1} \D, \\
&& (E_iE_j-E_jE_i) \D=0, \quad (F_iF_j-F_jF_i) \D = 0,\quad {\rm if}\ |i-j|>1,\\
&&(E_iE_iE_j-[2]_v E_iE_jE_i+E_jE_iE_i) \D=0, \quad \hspace{.75cm} {\rm if} \ |i-j|=1,\\
&& (F_iF_iF_j-[2]_v F_iF_jF_i+F_jF_iF_i) \D=0, \quad \hspace{1.1cm} {\rm if} \ |i-j|=1.
\end{eqnarray*}
\end{prop}

\begin{proof}
Proposition can be shown by using  (\ref{eq65}) and Proposition \ref{prop13}.
One could prove them directly by using the same argument as we make for Proposition \ref{prop13}.
More precisely,
all identities can be reduced into $\mcal S^m$ by replacing $[\A]$ by $[{}_p \A]$.
Proposition then follows from Proposition \ref{prop7}.
\end{proof}

\subsection{The identification  $\mcal K^m  = \dot{\U}^m$}

Recall the algebra $\U^m$ from Section ~\ref{Um}.
Following \cite[Section 23]{Lu93}, we shall define the modified form $\dot{\U}^m$ of $\U^m$.
We set
$$\Lambda^m=\{\lambda\in \mbb Z^N|\lambda_i = \lambda_{N+1-i}, \lambda_{n+1}=0\}.$$
For any $\lambda, \lambda' \in \Lambda^m$, we define
$$
{}_{\lambda}\U^m_{\lambda'}
= \U^m / (\sum_{a=1}^{n+1}(H_a-v^{-\lambda_a}) \U^m
+\sum_{a=1}^{n+1} \U^m  (H_a-v^{-\lambda_a})).
$$
Let $\pi_{\lambda, \lambda'}:\U^m \rightarrow {}_{\lambda} \U^m_{\lambda'}$ be the canonical projection.
We set $\sgn (\pi_{\lambda, \lambda} (J_+))=+$ and $\sgn (\pi_{\lambda, \lambda} (J_-) )= -$.
Set
$$
\dot{\U}^m  = \oplus_{\lambda,\lambda'\in \Lambda^m} {}_{\lambda} \U^m_{\lambda'}.
$$
Similarly, we can define $\dot{\U}^{\imath}$ by replacing $\U^m$ by its subalgebra $\U^{\imath}$. 
(See also ~\cite[5.6]{BKLW}.)
Following \cite[Section 23]{Lu93}, we have
\begin{equation} \label{Um-Ui}
\dot{\U}^m
=\sum_{\D} \U^m \D
=\sum_{\sgn (\D) =+} \U^{\imath} \D \oplus \sum_{\sgn (\D)=-}  \U^{\imath} \D
\simeq \dot \U^{\imath} \oplus \dot \U^{\imath},
\end{equation}
as vector spaces, where the sum runs over all elements $\D$ of the form $\pi_{\lambda, \lambda}(J_+)$ or
$\pi_{\lambda, \lambda } (J_-)$ for $\lambda\in \Lambda^m$.

Let $\mrm A_{\mbf D}$ be the associative   $\mbb Q(v)$-algebra without unit generated by $E_i \D, F_i \D$,
 $T \D$ and $\D$ for all $i\in [1,n-1]$ and $\D$ runs over all diagonal signed matrices in $\wt{\Xi}'_{\mbf D}$, subjects to
the  relations (i)-(viii) in Proposition \ref{prop14}. We have

\begin{prop}
\label{lem3}
The map $\phi: \mrm A_{\mbf D} \rightarrow \dot{\U}^m$ sending generators in $\mrm A_{\mbf D}$ to the respective elements in $\dot \U^m$  is an algebra isomorphism.
\end{prop}

\begin{proof}
Observe that all relations in $\U^m$ can be transformed into corresponding relations in $\dot{\U}^m$ by adjoining
diagonal signed matrixes.
By comparing the defining  relations of $\U^m$ and those in Proposition \ref{prop14}, we have that
$\dot \U^m$ is an associative  $\mbb Q(v)$-algebra generated by
$E_i \D$, $F_i \D$, $T \D$ and $\D$ for all $i\in [1,n-1]$ and $\D$ diagonal signed matrices in $\wt{\Xi}'_{\mbf D}$ and
subject to the defining   relations of $\mrm A_{\mbf D}$.
So we see that the map $\phi$ is a surjective algebra homomorphism.

It is  left to show $\phi$ is injective.
By using the same argument  of (\ref{Um-Ui}),
we have $\mrm A_{\mbf D}\simeq \dot{\mbf U}^{\imath}\oplus \dot{\mbf U}^{\imath}$, as vector spaces.
So the map $\phi$ is injective. We are done.
 \end{proof}

\begin{thm} \label{Km-present}
The assignment of sending generators in $\dot \U^m$ to the respective generators in $\mcal K^m$
defines an algebra isomorphism $\Upsilon': \dot{\U}^m \rightarrow \mbb Q(v)\otimes_{\mcal A}  \mcal K^m$.
\end{thm}

\begin{proof}
By Propositions ~\ref{prop14} and ~\ref{lem3}, we see that  $\Upsilon'$ is a surjective algebra homomorphism.
We observe that $\mbb Q(v)\otimes_{\mcal A}   \mcal K^m$
is a direct sum of two copies of $\dot \U^{\imath}$ as $\mbb Q(v)$ vector spaces. So we have the injectivity.
\end{proof}



\subsection{The algebra $\mcal U^m$}


Recall the algebra $\hat{\mcal K}$ and the notations $0^{\pm}$ from Section \ref{sec6}
and the notation $\hat{\A} (\mbf j)$ from (\ref{wtA}).
We consider the following elements in $\hat{\mcal K}$.
\begin{equation*}
\begin{split}
\mrm O (\mbf j) & =0^+(\mbf j)+0^-(\mbf j),\quad \forall \mbf j\in \mbb Z^N,\\
E_i &=E_{i+1,i}^{\theta, +}(0)+E_{i+1,i}^{\theta, -}(0),\\
F_i & =E_{i, i+1}^{\theta,+}(0)+E_{i, i+1}^{\theta,-}(0),\quad \forall i\in [1,n-1],\\
T &=
\textstyle
 \sum_{\lambda } ([D_{\lambda}^++E_{n,n+2}^{\theta}]+[D_{\lambda}^-+E_{n,n+2}^{\theta}]).
\end{split}
\end{equation*}
Let $\mcal U^m $ be the subalgebra of $\hat{\mcal K}$ generated by $E_i, F_i, T, \mrm O(\mbf j), 0^+(0)$ and
$0^-(0)$ for all $i\in [1,n-1]$ and $\mbf j\in \mbb Z^N$.
By a similar argument as  Proposition \ref{prop12}, we have the following proposition.

\begin{prop}\label{prop12}
The following relations hold in $\mcal U^m$.
\allowdisplaybreaks
\begin{eqnarray*}
&& \mrm O(\mbf j) \mrm O(\mbf j')= \mrm O(\mbf j') \mrm O(\mbf j),\ 0^{\pm}(0) \mrm O(\mbf j)= \mrm O(\mbf j)0^{\pm}(0),\ \ 0^{\alpha}(0)0^{\epsilon}(0)=\delta_{\alpha,\epsilon}0^{\alpha}(0),\\
&& \mrm O(\mbf j)F_h=v^{j_h-j_{h+1}}F_h  \mrm O(\mbf j),\
      \mrm O (\mbf j)E_h=v^{-j_h+j_{h+1}}E_h \mrm O(\mbf j),\ \\
&& E_hT=TE_h,\quad F_hT=TF_h,\quad \hspace{2cm} {\rm if} \ h\in [1,n-2],\\
&&0^{\pm}(0)E_h=E_h 0^{\pm}(0),\  F_h0^{\pm}(0)=0^{\pm}(0)F_h,\ \mrm O(\mbf j)T=T \mrm O(\mbf j),\ 0^{\pm}(0)T=T0^{\mp}(0),\\
&&F_hE_h-E_hF_h=(v-v^{-1})^{-1}(0( \underline h- \underline{h+1})-0( \underline{h+1} - \underline h))\\
&&E_{n-1}^2T-[2]_v E_{n-1}TE_{n-1}+TE_{n-1}^2=0,\\
&&F_{n-1}^2T-[2]_v F_{n-1}TF_{n-1}+TF_{n-1}^2=0,\\
&&T^2E_{n-1}-[2]_v TE_{n-1}T+E_{n-1}T^2=E_{n-1},\\
&&T^2F_{n-1}-[2]_v TF_{n-1}T+F_{n-1}T^2=F_{n-1},\\
&&E_iE_j=E_jE_i,\quad F_iF_j=F_jF_i,\quad \hspace{1.8cm} {\rm if}\ |i-j|>1,\\
&&E_i^2E_j-[2]_v E_iE_jE_i+E_jE_i^2=0,\quad \hspace{1.3cm}{\rm if}\ |i-j|=1,\\
&&F_i^2F_j-[2]_v F_iF_jF_i+F_jF_i^2=0,\quad \hspace{1.5cm}  {\rm if}\ |i-j|=1.
\end{eqnarray*}
\end{prop}

By comparing the defining relations and graded dimensions, we have

\begin{cor}
We have a unique  isomorphism $\U^m \to \mcal U^m$
defined by  $E_i\mapsto E_i$, $F_i\mapsto F_i$, $T\mapsto T$, $H_a \mapsto {\rm O}(- \underline a)$
 and $J_{\pm}\mapsto 0^{\pm}(0)$,  for any $i\in [1,n-1]$ and $a\in [1, n]$.
\end{cor}

\end{document}